%% file: LeonoriMagliocca.tex
\documentclass[10pt,twoside,english,reqno,a4paper]{amsart}
\input{style.tex}

\author[T. Leonori]{Tommaso Leonori}
\address[T. Leonori]{Dipartimento di Scienze di Base e Applicate per l'Ingegneria\\  "Sapienza" \ Universit\`a di Roma I\\ Via Antonio Scarpa 10,   00161, Roma, Italy}
\email{tommaso.leonori@sbai.uniroma1.it}

\author[M. Magliocca]{Martina Magliocca}

\address[M. Magliocca]{Dipartimento di Matematica, Universit\`a degli Studi Tor Vergata, Via della Ricerca Scientifica 1, 00133 Rome, Italy. }
 \email{magliocc@mat.uniroma2.it}

\keywords{Uniqueness, Nonlinear Parabolic Equations, Unbounded Solutions, Nonlinear Lower Order Terms} \subjclass[2010]{35B51, 35K55,35K61}

\begin{document}

\title[Comparison results for unbounded solutions for a parabolic Cauchy-Dirichlet problem]{Comparison results for unbounded solutions for a parabolic Cauchy-Dirichlet problem with superlinear gradient growth}

\begin{abstract}
In this paper we deal with uniqueness of solutions to the following problem 
\begin{equation*}
\begin{cases}
\begin{array}{ll}
\ds u_t-\D_p u=H(t,x,\N u)&\ds\text{in}\quad \ Q_T,\\
\ds u (t,x) =0 & \ds\text{on}\quad(0,T)\times \partial \Omega,\\
 \ds u(0,x)=u_0(x) &\ds\text{in }\quad \Omega,
\end{array}
\end{cases}
\end{equation*}
where $Q_T=(0,T)\times \Omega$ is the parabolic cylinder, $\Omega$ is an open subset of $\mathbb{R}^N$, $N\ge2$, $1<p<N$,  and the right hand side  $\ds H(t,x,\xi):(0,T)\times\Omega \times \mathbb{R}^N\to \mathbb{R}$ exhibits a superlinear growth with respect to the gradient term.

\end{abstract}

\maketitle

%\tableofcontents

\section{Introduction}
\setcounter{equation}{0}
\renewcommand{\theequation}{\thesection.\arabic{equation}}
\numberwithin{equation}{section}

The present paper is devoted to the study of the uniqueness and, more in general, to the  comparison principle between sub and supersolutions {{of}} nonlinear parabolic problems with lower order terms that have at most a power growth with respect to the gradient. More specifically, we set $\Omega$   a bounded  open subset of $\rn $, with $N\geq 3$, and   $T>0$. We consider a   Cauchy--Dirichlet problem of the type 
 \begin{equation}\label{model1}
\begin{cases}
\begin{array}{ll}
\ds  u_t- \Delta_p u=h  (t,x,\N u) + f(t,x) &\ds\text{in}\quad Q_T,\\
\ds  u (t,x) =0 &\ds\text{on}\quad(0,T)\times \partial \Omega,\\
\ds  u(0,x)=u_0(x) &\ds\text{in}\quad \Omega,
\end{array}
\end{cases}
\end{equation}
where $Q_T=(0,T)\times \Omega$ denotes the parabolic cylinder, $-\Delta_p  $ is the usual $p-$Laplacian with $p>1$, the functions $u_0$ and $f$ belong to  suitable Lebesgue spaces and 
$h  (t,x, \xi)$ is  a Cartheodory function that has (at most) $q-$growth with respect to the last variable, being $q$ \lq\lq superlinear\rq\rq{ } and  smaller than $p$.

The model equation we have in mind is the following 
 \begin{equation}\label{model}%%%%\tag{$\D_p$}
\begin{cases}
\begin{array}{ll}
 \ds u_t- \Delta_p u=|\N u|^q + f(t,x) &\ds\text{in}\quad Q_T,\\
 \ds u (t,x) =0 &\ds\text{on}\quad(0,T)\times \partial \Omega,\\
 \ds u(0,x)=u_0(x) &\ds\text{in}\quad \Omega,
\end{array}
\end{cases}
\end{equation}
for $1<q<p$,  $f\in L^r(0,T;L^m(\Omega))$, for some $r,m \geq 1$, and $u_0 \in L^{s} (\Omega)$, for some $s\geq1$.

\medskip 

The literature about comparison principles for weak sub/super solutions of \eqref{model} is mainly devoted to cases in which solutions are smooth (say for instance continuous),  the equation is exactly the one in \eqref{model} or the growth of the nonlinear term is \lq\lq sublinear\rq\rq. 

Our aim is to generalize this kind of results to the case of \emph{unbounded solutions} and \emph{non regular data} (both the initial datum and the forcing term), dealing with sub/supersolutions in a suitable class. 

\medskip 

Let us mention that in the elliptic framework such a kind of results have been studied in several papers using different techniques. Let us recall the papers \cite{ABM},  \cite{BaM} \cite{BDMP}, \cite{BMMP}, \cite{BMMP2},   \cite{Me} (and references cited therein) where unbounded solutions for quasilinear equations have been treated. We want also to highlight the results of \cite{Po} (see also\cite{BaP}, \cite{LPR} and \cite{LP})  that have inspired our work,  where the comparison principle among unbounded sub/supersolutions  has been proved, for sub/supersolutions that have a suitable power that belongs to the energy space.

Let us also mention that, as well explained in \cite{ADaP} (see also \cite{ADaP2} for the parabolic counterpart) things change drastically when one deals with the so called {\it natural growth} (i.e. $q=p$ in \eqref{model}), since in this case the right class in which looking for uniqueness involves a suitable exponential of the solution (one can convince   himself just by performing the Hopf-Cole transformation to the equation in \eqref{model}). 

\medskip 

The literature is much poorer in the parabolic case, especially when unbounded solutions are considered. 
Let us mention the results in 
\cite{Feo},\cite{DNFG} where nonlinear problem of the type \eqref{model1} are considered where $h(t,x,\xi)$ has a sublinear (in the sense of the $p$-Laplacian type operators, see \cite{Ma} for more details about such a threshold) growth with respect to the last variable. 

\medskip 

In order to prove the comparison principle (that has the uniqueness has byproduct), several techniques have been developed. Let us mention, among the others, the results that have been proved by using the monotone rearrangement 
technique (see for instance \cite{BDMP}    and references cited therein) and by means of viscosity solutions (see for instance \cite{CIL}, \cite{BDL} and references cited therein). 

Our choice, that has been mainly inspired by \cite{Po}, uses both an argument via linearization and a method that exploits a sort of    convexity of the hamiltonian term with respect to the gradient. These two approaches are, in some sense, complementary since {{the first} one (the linearization) works in the case $1<p\leq 2$ while the {{second one (the \lq\lq\,convex\rq\rq one)} deals with $p\geq2$. Of course, the only case in which both of them are in force is when $p=2$.

\medskip 

Since we want to deal with unbounded solutions and irregular data, the way of defining properly the sub/supersolutions is through the renormalized formulation (see \cite{BlM}, \cite{pe}, \cite{BlP} and \cite{PPP}). 

The renormalized formulation, that is the most natural one in this framework, is helpful in order to face the first difficulty of our problem, that is the unboundedness of the sub/supersolutions. Indeed we can  decompose   the sub/supersolutions into their bounded part plus a reminder that can be estimated, using the uniqueness class  we are working in. 

\medskip 
According with the results in the stationary case, we prove that the uniqueness class (i.e. the class of functions for which we can prove the comparison principle, and uniqueness as a byproduct) is the set of functions whose a  suitable power $\gamma$ (that depends only on $q$, $p$ and $N$) belongs to the energy space. Let us recall (see \cite{Ma}) that such a class is also the right one in order to have existence of solutions. 

Even more, we show, through a counterexample, that at least for $p=2$, the class of uniqueness is the right one, adapting an argument of \cite{BASW1}--\cite{BASW2} to our case. 

\medskip 

We first consider the case with  $1<p\leq 2$, and we look for an inequality solved by the difference between the bounded parts of the sub and supersolutions, using the linearization of the lower order term. Let us recall that this is the typical approach for   singular (i.e. $p\leq 2$) operators, that has extensively used in several previous papers (see for instance \cite{Feo} and \cite{ABM} and references cited therein).  In this case we are allowed to deal with general Leray-Lions operators, even if, due to a lack of regularity of the the sub/supersolutions,   we cannot cover all the superlinear and subnatural  growths.

\medskip

The second part of the paper is devoted to  the case  $p\geq 2$ that  is, in some way,  more complicated, due to the degenerate nature of the operator. 
In fact, we need to straight the hypotheses on the differential operator considering a perturbation (through a  matrix with bounded coefficients) of the standard p-Laplacian. 
%In order to deal with this case, we use an approach that 

 Here the idea is to perturb  the difference between the bounded parts of the sub and supersolutions
 and to  exploit the convexity of the lower order term with respect to the gradient (at least in the case $p=2$, otherwise the general hypothesis is more involved).

\medskip

The plan of the paper is the following: in Section 2 we collect all the statement of our results, while Section 3 is devoted to some  technical results. The proofs of the main results are set in Section 4, if $1<p\leq 2$ and in Section 5 if $p\geq 2$. 

Finally in the Appendix  there is an example that shows that the uniqueness class is the right one, at least for $p=2$ and $1<q<2$.

\section{Assumptions and statements of the results}
\setcounter{equation}{0}
\renewcommand{\theequation}{\thesection.\arabic{equation}}
\numberwithin{equation}{section}

As already explained in the Introduction, we deal with the following Cauchy-Dirichlet problem:
\begin{equation}\label{Dp} %\tag{$\mathbbm{P}$}
\begin{cases}
\begin{array}{ll}
\ds  u_t-\dive a(t,x,\N u)=H(t,x,\N u)&\ds\text{in}\quad Q_T,\\
 \ds u (t,x) =0 &\ds\text{on}\quad(0,T)\times \partial \Omega,\\
\ds  u(0,x)=u_0(x) &\ds\text{in}\quad \Omega, 
\end{array}
\end{cases}
\end{equation}
with   $u_0 \in L^{1}(\Omega)$. 

\medskip 

The main assumptions on the functions involved in \eqref{Dp} are the following: the vector valued function $a(t,x,\xi):(0,T)\times \Omega\times\R^N\to \R^N$ is a Carath\'{e}odory function such that
%\begin{subequations}
%\makeatletter 
%\def\@currentlabel{$\t{A}_L$}
%\makeatother
%\label{linea}
%\renewcommand{\theequation}{$\t{A}_L$\arabic{equation}}
\begin{equation}\label{a2}  
\exists \alpha>0:\quad\left(a(t,x,\xi)-a(t,x,\eta)\right)\cdot (\xi-\eta)\ge
\alpha(|\xi|^2+|\eta|^2)^{\frac{p-2}{2}}|\xi-\eta|^2,  
\end{equation}
\begin{equation}\label{a3} 
%\begin{array}{c}
%\ds
\exists \beta>0:\quad |a(t,x,\xi)|\le \beta\left[\ell (t,x)+ |\xi|^{p-1}\right],\quad\ell\in L^{p'}(Q_T)\,,
\end{equation}
\begin{equation}\label{a4}   
a(t,x,0)=0\,.
\end{equation}
%\end{subequations}
for a.e. $ (t,x)\in Q_T$, $\forall \eta,  \xi\in \rn$, with $1<p<N$. 

As far as the lower order term is concerned, we suppose that $H(t,x, \xi):(0,T)\times \Omega\times\R^N\to \R^N$ is  a Carath\'{e}odory function that  satisfies 
the following \emph{growth condition}:
\begin{equation}\label{H1}%%\tag{H}
\begin{array}{c}
\ds
\exists c_1>0:\quad|H(t,x,\xi)|\le c_1|\xi|^q+f\quad \mbox{with} \quad  \max\left\{\frac{p}2, \frac{p(N+1)-N}{N+2}\right\} < q<p \,,
\end{array}
\end{equation}
%\footnote{\hl{q=max... e`ammissibile? ogni tanto si include, ogni tanto no...\\ discusso, da aggiungere rmk lin/sublin e togliere inclusione nel caso p=2}}
for a.e. $ (t,x)\in Q_T$, $\forall  \xi\in \rn$
 and with 
$f=f(t,x)$ belonging to  some Lebesgue space.

\medskip

%We observe that the upper bound in \eqref{H1} (i.e. the case $q=p$) is the so-called {\it natural growth} condition. In our study we are interested in  sub-natural growths as well as in \lq\lq superlinear\rq\rq ones. The superlinear threshold for operators of $p$-Laplacian type relies to be, as well explained in \cite{Ma}, the value $q=\max\left\{\frac{p}2, \frac{p(N+1)-N}{N+2}\right\} $. \\
%\blu{Comparison results regarding both the \emph{sublinear} $q< \max\left\{\frac{p}2, \frac{p(N+1)-N}{N+2}\right\} $ and the \emph{linear} $q= \max\left\{\frac{p}2, \frac{p(N+1)-N}{N+2}\right\} $ cases are included in our paper. Indeed, it is sufficient to observe that Young's inequality provides us with\[
%|\N u|^{q_1}\le c|\N u|^{q_2} +c
%\]
%where 
%$q_1$ is, at most, linear and $q_2$ is superlinear. 
% }
\medskip

First of all we need to determine the meaning of sub/supersolutions we want to deal with. Since we are interested in possibly irregular data and, in general, in unbounded solutions, the most natural way to mean sub/supersolutions is trough the renormalized formulation.  
In order to introduce such an issue, we {first need} to  define a natural space where such sub/supersolutions are defined: taking inspiration from  \cite{BBGGPV},  we  set  
\[
\mathcal{T}^{1,p}_0  (Q_T)= \left\{ u:Q_T\to \mathbb{R}\,\, \t{ a.e. finite} : \,\,   {T}_k  (u) \in  L^p \big(0,T;W^{1,p}_0 (\Omega)\big) \right\}, \q \t{ for}\,\,p\ge 1,
\] 
where    $T_k (s)= \max\{-k,\min\{k,s \}\}$, for $k\geq 0$ and $s\in \R$. 

\medskip

Now we are ready to define the renormalized sub/super solutions to \eqref{Dp}. 

\medskip

\begin{definition}\label{defrin2}
We say that a function $u\in \mathcal{T}^{1,p}_0 (Q_T) $ is a renormalized subsolution (respectively a supersolution) of \eqref{Dp} if
\begin{equation*}
 H(t,x,\N u)\in L^1(Q_T),\quad u\in C([0,T];L^1(\Omega))
\end{equation*}
and  it satisfies: 
\begin{equation}\label{renf}
\begin{array}{c}
\ds
-\int_\Omega S(u_0)\varphi(0)\,dx+\iint_{Q_T} \left[-S(u)\varphi_t+ S'(u)a(t,x,\N u)\cdot \N\vp+S''(u)a(t,x,\N u)\cdot \N u\,\vp\right]\,dx\,dt\\
[4mm]\ds
\le(\ge) \iint_{Q_T} H(t,x,\N u)S'(u)\varphi\,dx\,dt,
\end{array}
\end{equation}
with 
\begin{equation*}
u(t,x)\bigl|_{t=0} \leq(\ge) u_0(x)\qquad\text{in}\,\,L^1(\Omega)\,,  
\end{equation*}
for every $S\in W^{2,\infty}(\mathbb{R})$ such that $S'(\cdot)$ is nonnegative, compactly supported and for every 
\[
0\le\vp\in L^{\infty}(Q_T))\cap L^p(0,T;W_0^{1,p}(\Omega)),\,\, \vp_t\in L^1(Q_T)+L^{p'}(0,T;W^{-1,p'}(\Omega)),\,\,  \vp(T,x)=0.
\]
\end{definition}

Some remarks about the above definition are in order to be given.

\begin{remark}
 Let us observe that usually the renormalized formulation is   equipped with an additional condition about the asymptotic behavior of the energy, i.e. 
 \begin{equation}\label{ET}
\lim_{n\to\infty}\frac{1}{n}\iint_{\{n<|u|<2n\}}|\N u|^p\,dx\,dt=\lim_{n\to\infty}\frac{1}{n}\iint_{\{n<|v|<2n\}}|\N v|^p\,dx\,dt=0\,.
\end{equation}
Such a  condition is required in order to guarantee  that renormalized solution are, in fact, distributional ones. 
In our case we do not have to ask, in general,   \eqref{ET} to hold  since it is a consequence of the class of uniqueness that we consider (see \cref{lemgrad}). More specifically, we have to impose such a condition only in the case in which we deal with $L^1$-data and with \lq\lq low\rq\rq{ } values of $q$. 
\end{remark}

\begin{remark}%\footnote{\red {io la spezzerei in due remark, i e ii sono disconnesse tra di loro...}}
\begin{itemize} 
\item[i)] Note  that a subsolution (a supersolution) on $Q_T$ turns out to be a subsolution (a supersolution) on $Q_t$ for any $0<t\le T$. Thus,   with an abuse of notation,   we   refer to \cref{defrin2} even if we take into account \eqref{renf} evaluated over $Q_t$, with $0<t\le T$.

\item[ii)]  For renormalized solutions of an equation of the type \eqref{model}, the regularity $u\in C([0,T];L^1(\Omega))$ is deduced directly by the  renormalized formulation,  via a trace result (see \cite{P1}). However, since we are dealing with sub/supersolutions, we need to add it to the definition. 
\end{itemize}
\end{remark}

\subsection{Assumptions for $p=2$}
As already announced in the Introduction, for problem \eqref{Dp} with $p=2$ we can use both the approach by linearization and the one by \lq\lq convexity\rq\rq. 

\medskip 

The first approach we want to deal with is the one by linearization. Hence we assume that 
$a(t,x,\xi):(0,T)\times \Omega\times\R^N\to \R^N$ satisfies \eqref{a2}--\eqref{a4}, that in this particular case read as:
\begin{equation}\label{linea23}  
\exists \alpha>0:\quad\left(a(t,x,\xi)-a(t,x,\eta)\right)\cdot (\xi-\eta)\ge\alpha |\xi-\eta|^2
\end{equation}
\begin{equation}\label{linea21}
\exists \b>0:\quad |a(t,x,\xi)|\le \beta\left[\ell (t,x)+ |\xi|\right]\quad\text{for}\quad\ell\in L^{2}(Q_T),
\end{equation}
\begin{equation}\label{linea22} 
a(t,x,0)=0,
\end{equation}
a.e. $(t,x)\in Q_T$, for all $\xi, \eta \in \mathbb{R}^N$.

Moreover we {{assume}} the 
growth assumption
\begin{equation}\label{H111}%\tag{H}
\exists \cu>0:\q
|H(t,x,\xi)|\le \cu |\xi|^q+f(t,x)\quad \mbox{for } \quad 1 \leq  q<2, 
\end{equation}
a.e. in  $(t,x)\in Q_T,\,\, \forall \xi\in \R^N$. In addition we suppose    the following locally (weighted) Lipschitz condition
\begin{equation}\label{Hprime}%\tag{$\mathbbm{h}$}
\begin{array}{c}
\ds
\exists \cd>0:\quad|H(t,x,\xi)-H(t,x,\eta)|\le \cd|\xi-\eta|\left[g(t,x)+|\xi|^{q-1}+|\eta|^{q-1}   \right]
\end{array}
\end{equation}
is in force,   for some function $g(t,x)$ belonging to a suitable Lebesgue space we specify later. 

\smallskip

Before stating our  comparison results, we need to introduce the class of uniqueness. 
As for the elliptic case (see \cite{Po,BaP} and also \cite{GMP1}--\cite{GMP2}), the right framework is the set of sub/supersolutions $u,v$ whose power $\gamma=\ga(q)$ belongs to the energy space for a  suitable choice of $\gamma$. 
Moreover we consider the initial  data $u_0,v_0$ belonging to $L^\sigma (\Omega) $ for some $\sigma \geq 1$. More precisely, we consider sub/supersolutions satisfying
\begin{equation}\label{csi}%\tag{cr}
u,\, v\in C([0,T];L^{\si}(\Omega))\qquad
{\text{with}} \quad  \sigma= \frac{N\big(q-1\big)}{2-q}
\end{equation}
and
\begin{equation}\label{potc}%\tag{$\mathbbm{uc}$}
\begin{array}{c}
\ds
(1+|u|)^{\frac{\si}{2}-1}u,\,(1+|v|)^{\frac{\si}{2}-1}v\in L^2(0,T;H^{1}_0(\Omega)).
\end{array}
\end{equation}

Such a class of uniqueness makes sense whenever $1\leq \sigma$, i.e. if $q\geq 2-\frac{N}{N+1}$. 

\begin{remark}
One can convince himself that the uniqueness class is the right one just by constructing a counterexample of a problem of the type  \eqref{Dp} that admits (at least) two solutions, whose just one belongs to the right class. The construction of such a pair of solutions is a bit involved and we left it to the \cref{app}. 

\end{remark}

The assumptions about the data are strictly related to the value of the superlinearity $q$. For this reason, we split the superlinear growth   of the gradient term into two   subintervals for which, in turn, we require two different compatibility conditions on the data and two different class of uniqueness. 

\smallskip

We first consider  superlinear rates belonging to the range 
\begin{equation}\label{q1prime}%\tag{$\mathbbm{q}_1$}
\begin{array}{c}
\ds
2-\frac{N}{N+1}< q\le 2-\frac{N}{N+2}
\end{array}
\end{equation}
that correspond to the case $1<\sigma\leq 2$, and that allows us to consider  $f(t,x)\in L^r(0,T;L^m(\Omega))$ in \eqref{H111} that verifies 
\begin{equation}\label{F11}%\tag{$\mathbbm{f}_{r,m}$}
m\ne \infty,\,r\ne\infty\quad\text{s.t.}\quad\frac{N\sigma}{m}+\frac{2\sigma}{r}\le N+2\sigma\quad
%\text{with}\quad \si=\frac{N(q-1)}{2-q}
\end{equation}
while     $g(t,x)\in L^d(Q_T)$ in \eqref{Hprime} satisfies 
\begin{equation}\label{G1}%\tag{$\mathbbm{{g}}_1$}
g\in L^{d}(Q_T)\quad\text{with}\quad d= N+2.
\end{equation}

\medskip 

Our first result is the following.

\begin{theorem} \label{thm25}
Assume that $a(t,x,\xi)$ satisfies \eqref{linea23}--\eqref{linea22}, $H(t,x,\xi)$ satisfies  \eqref{H111} \eqref{Hprime}     and   that \eqref{q1prime}--\eqref{G1} hold true. Let $u$ and $v$ be a renormalized subsolution and a supersolution of \eqref{Dp}, respectively,  satisfying \eqref{csi}, \eqref{potc} and let $u_0,\,v_0\in L^\si(\Omega)$  such that  $u_0\le v_0$. Then  $u\le v$ in ${Q_T}$.
\end{theorem}

\begin{remark}%\footnote{questa la metterei dopo la prova del teorema 2.11 e 2.5, così si capisce meglio}
As far as the limit  the case $q=2-\frac{N}{N-1}$ is concerned, we observe that the result of \cref{thm25} still holds true assuming the data 
\[
u_0\in L^{1+\om}(\Omega),\, f\in L^{1+\om}(Q_T), \qquad \forall \omega>0, 
\]
for sub/supersolutions $u,v$ that belong to the class 
\begin{equation*} 
\begin{array}{c}
\ds
(1+|u|)^{\frac{\om-1}{2}}u,\,(1+|v|)^{\frac{\om-1}{2}}v\in L^2(0,T;H^{1}_0(\Omega)).
\end{array}
\end{equation*}
The proof follows as the one of \cref{thm25} with minor changes, so we omit it. 
\end{remark}

%\medskip 
%
%\red{We refer the reader to \cref{lost2} to motivate the loss of the threshold $2-\frac{2}{N+2}<q<2$.}

\medskip

Secondly  consider the range given by 
\begin{equation}\label{q2prime}%\tag{$\mathbbm{{q}}_2$}
1\leq q< 2-\frac{N}{N+1}
\end{equation}
that correspond to $\sigma<1$, and we 
require that the functions $f$ and $g$ satisfy
\begin{equation}\label{F22}%\tag{$\mathbbm{f}_1$}
f\in L^1(Q_T)
\end{equation}
and
\begin{equation}\label{G2}%\tag{$\mathbbm{{g}}_2$}
g\in L^d(Q_T)\quad\text{with}\quad d>N+2.
\end{equation}

\medskip 

Thus our result in this framework is the following.

\begin{theorem}  \label{thm26}
Assume that $a(t,x,\xi)$ satisfies \eqref{linea23}--\eqref{linea22}, $H(t,x,\xi)$ satisfies  \eqref{H111} \eqref{Hprime}     and   that \eqref{q2prime}--\eqref{G2} hold true. Let $u$ and $v$ be a renormalized subsolution and a supersolution of \eqref{Dp}, respectively,  satisfying \eqref{ET}   and let $u_0,\,v_0\in L^1(\Omega)$  such that  $u_0\le v_0$. Then  $u\le v$ in ${Q_T}$.
\end{theorem}

Let us observe that, in fact, our results do not cover all the interval $1\leq q<2$. This is due to a lack of regularity of the sub/super solutions (see Remark \ref{lost2} below).

\medskip

The second approach to the comparison principle deals with a trick that uses the convexity of the lower order order term. Such a method is not as robust as the linearization one, so we need to strength the hypotheses on the differential operator. 

We consider here the following problem 
\begin{equation}\label{p2} 
\begin{cases}
\begin{array}{ll}
u_t-\dive \big(A(t,x)\N u \big)=H(t,x,\nabla u)&\text{in}\,\, Q_T,\\
u (t,x) =0 & \text{on}\,\, (0,T)\times \partial \Omega,\\
u(0,x)=u_0(x) &\text{in}\,\, \Omega\,.
\end{array}
\end{cases}
\end{equation}

We assume that $A:(0,T)\times\Omega\to\mathbb{R}^{N\times N}$ is a bounded and uniformly elliptic matrix with measurable coefficients, i.e. 
\begin{equation}\label{aa1}
\begin{array}{c}\ds
A(t,x)=\big\{a_{ij}(t,x)\big\}_{i,j=1}^N \quad\text{with} \quad a_{ij}\in L^{\infty}(Q_T) \quad\forall  1\le i,j\le N, 
\\[4mm]\ds
\mbox{such that } \q \exists  \al,\,\b:\q 0<\al\le \beta \quad\text{and} \quad \al |\xi|^2\le A(t,x)\xi\cdot\xi\le \beta |\xi|^2\,,
\end{array}
\end{equation}
for almost every $(t,x)\in Q_T$ and for every $\xi\in \R^N$,

\medskip 

As far as  the lower order term is concerned, we suppose that the nonlinear term $H(t,x,\xi)$ satisfies \eqref{H111} with $1<q<2$ and it  can be decomposed as 
\begin{equation}\label{dec}  
H(t,x,\xi)=H_1(t,x,\xi)+H_2(t,x,\xi)
\end{equation}
where, for a.e. $(t,x)\in Q_T$ and for every $\xi$, $\eta$ in $ \mathbb{R}^N$,    the functions $H_1(t,x,\xi)$ and $H_2(t,x,\xi)$ verify: 
\begin{itemize}
\item $ H_1(t,x,\xi):(0,T)\times \Omega\times\R^N\to \R$ is a convex function with respect to the  $\xi$ variable, i.e.
\begin{equation}\label{H1.1}  
\forall \varepsilon \in (0,1) \quad H_1\left(t,x,\varepsilon\xi+(1-\varepsilon)\eta\right)\le \varepsilon H_1(t,x,\xi)+(1-\varepsilon)H_1(t,x,\eta)\quad\forall \xi,\eta \in \mathbb{R}^N\,;
\end{equation}
% which satisfies the growth assumption in \eqref{H111} with $1<q<2$;

\item   $H_2(t,x,\xi):(0,T)\times \Omega\times\R^N\to \R$ is a Lipschitz function with respect to $\xi$, namely
\begin{equation}\label{H2.1}  
\exists c_2>0:\quad|H_2(t,x,\xi)-H_2(t,x,\eta)|\le c_2|\xi-\eta| 
\end{equation}
that satisfies the following inequality    for sufficiently small $\eps>0$
\begin{equation}\label{H2.2} 
H_2(t,x,(1-\varepsilon)\xi)-(1-\varepsilon)H_2(t,x,\xi)\le 0 
\end{equation}
for almost every $(t,x)\in Q_T$ and for all $\xi,\,\eta \in \mathbb{R}^N$.
\end{itemize}

\bigskip

As for the approach by linearization, we have two types of results, depending on the regularity of the sub/supersolutions under consideration. 

First we deal with solutions in the class \eqref{csi}--\eqref{potc}: in this case we consider 
lower order terms whose growth with respect to  the gradient is at most a power $q$ in the range 
\begin{equation}\label{q1pprime}%\tag{${\text{{q}}}_1$}
2-\frac{N}{N+1}< q<2\,,
\end{equation}
and  we assume that $f$ in \eqref{H111}  belongs to $ L^r(0,T;L^m(\Omega))$ with  $(m,r)$ such that  \eqref{F11} holds true.

\medskip

Hence we have the following result.

\begin{theorem}\label{comp2}
Assume  that $A(t,x)$ satisfies  \eqref{aa1} and $H(t,x,\xi)$ \eqref{H111} together with \eqref{dec}--\eqref{H2.2}, \eqref{q1pprime} and \eqref{F11},  and let $u$ and $v$ be, respectively, a renormalized subsolution and a supersolution of \eqref{p2} satisfying \eqref{csi}--\eqref{potc}. 
and let $u_0,\,v_0\in L^{\sigma} (\Omega)$ be such that  $u_0\le v_0$. Then  we have that $u\le v$ in ${Q_T}$.
\end{theorem}

\medskip

As far as the low values of $q$ are considered, we deal with the same range considered in \eqref{q2prime}
%\begin{equation}\label{q2pprime}%\tag{${\text{{q}}}_2$}
%1<q< 2-\frac{N}{N+1}
%\end{equation} 
and $L^1$ data.

\medskip

\begin{theorem}\label{teorinL12}
Assume  that $A(t,x)$ satisfies  \eqref{aa1} and $H(t,x,\xi)$ \eqref{H111} together with \eqref{dec}--\eqref{H2.2}, \eqref{q2prime} and \eqref{F22}. Let $u$ and $v$ be a renormalized subsolution and a supersolution of \eqref{p2}, respectively, satisfying \eqref{ET} and let  $u_0,\,v_0\in L^1(\Omega)$ with $u_0\le v_0$. Then  we have that $u\le v$ in ${Q_T}$.
\end{theorem}

\medskip

\subsection{Assumptions for $1<p< 2$}
Le us now go back to our original problem  \eqref{Dp}  and let us assume that the vector valued function $a(t,x,\xi)$ satisfies \eqref{a2}--\eqref{a4}, with $1<p< 2$. 

As far as the lower order term is concerned, we suppose that it satisfies 
the \emph{growth condition}  \eqref{H1} and 
   we assume that  a suitable weighted Lipschitz assumption with respect to the last variable is in force, i.e. 
\begin{equation}\label{H3}  
\begin{array}{c}
\ds
\exists c_2 >0:\quad|H(t,x,\xi)-H(t,x,\eta)|\le c_2 |\xi-\eta|\left[g+|\xi|^{q-1}+|\eta|^{q-1}   \right]
\end{array}
\end{equation}
for a.e. $(t,x)\in Q_T$, for all $\xi, \eta \in \mathbb{R}^N$ and for some measurable function $g=g(t,x)$ belonging to $ L^d (Q_T)$, for a suitable choice of $d\geq 1$. 
\medskip 

In this setting, we determine two ranges of $q$ each of them giving a different type of result in function of the required class of uniqueness (and the regularity) of the solutions. \\

We start by considering 
\begin{equation}\label{qu}  
\begin{array}{c}
\ds
p-\frac{N}{N+1}< q\le p-\frac{N}{N+2}\quad\text{for}\quad 1+\frac{N}{N+1}<p<2\\
[4mm]\ds
\text{and}
\\
\ds
1\le q<p-\frac{N}{N+2}\quad\text{for}\quad 1+\frac{N}{N+2}\le p< 1+\frac{N}{N+1}
\end{array}
\end{equation}
together with 
\begin{equation}\label{h1} 
g\in L^{d}(Q_T)\quad\text{with}\quad d=  \frac{N(q-(p-1))-p+2q}{q-1}. % \quad (d=\infty\,\,\text{if}\,\, q=1)\footnote{serve?}. 
\end{equation}
Moreover we assume that 
\begin{equation}\label{F1}%%\tag{$\t{F}_{r,m}$}
\begin{array}{c}
\ds
 f\in L^r(0,T;L^m(\Omega))\q\t{with}\q (m,r) \q\t{such that}\\ 
[2mm]\ds
m\ne\infty,\,r\ne\infty,\quad \frac{N\si}{m}+\frac{N(p-2)+p\si}{r}\le N(p-1)+p\si.
\end{array}
\end{equation}

\medskip

When we deal with this range, we  assume the continuity regularity
\begin{equation}\label{cont} 
u,\,v\in C([0,T];L^{\si}(\Omega))\q\t{with}\,\q \si=\frac{N(q-(p-1))}{p-q}
\end{equation}
and we deal with   the uniqueness class
\begin{equation}\label{pot} 
\begin{array}{c}
\ds
(1+|u|)^{\gamma-1}u,\,(1+|v|)^{\gamma-1}v\in L^p(0,T;W^{1,p}_0(\Omega))\q
\mbox{\text{with}} \q \gamma=\frac{\si+p-2}{p}.
\end{array}
\end{equation}
 In this ranges   of values of $q$ we have that $\sigma\in (1,2)$.
 
 \medskip

\begin{theorem}\label{teorin2}
Assume   \eqref{a2}--\eqref{a4},   and    that $H(t,x,\xi)$ satisfies \eqref{H1},  \eqref{H3} with  \eqref{qu}--\eqref{F1}. 
Let $u$ and $v$ be a subsolution and a supersolution of \eqref{Dp}, respectively,  
satisfying \eqref{cont}--\eqref{pot}. Then  $u\le v$ in ${Q_T}$.
\end{theorem}

The next range corresponds to the case of lower values of $q$.   Namely,  we consider
\begin{equation}\label{q3} 
1\le q < p-\frac{N}{N+1}\quad\text{and}\quad 1+\frac{N}{N+1}< p<2
\end{equation}
and  that $g$ fulfils
\begin{equation}\label{h2} 
g\in L^d(Q_T)\quad\text{with}\quad d>\frac{p(N+1)-N}{p(N+1)-(2N+1)}\quad \left(d= \infty\,\,\text{if}\,\, p=1+\frac{N}{N+1}\right).
\end{equation}
As far as the source term $f$ is concerned, we suppose 
\begin{equation}\label{F2}%%\tag{$\t{F}_1$}
f\in L^1(Q_T).
\end{equation}

\medskip

Thus the result is the following.

\medskip

\begin{theorem}\label{teo33}
Assume \eqref{a2}--\eqref{a3} and    that $H(t,x,\xi)$ satisfies  \eqref{H1}  and \eqref{H3} with    \eqref{q3}--\eqref{F2}. Let $u$ and $v$ be a renormalized subsolution and a renormalized supersolution of \eqref{Dp}, respectively, %in the sense of \cref{defrin2}
such that \eqref{ET} holds true. Then, we have that $u\le v$ in ${Q_T}$.
\end{theorem}

\begin{remark}\label{lost2} 
Let us mention some peculiarity of our results, for $1<p\leq 2$. 
\begin{itemize} 
\item[i)] It is worth pointing out that the case $p-\frac{N}{N+2}< q<p$ is not considered here. Indeed, as already observed in the elliptic case (see \cite[Remark $3.5$]{Po}), we would need to require more regularity on the gradient of the sub/supersolutions,  in order to apply the linearization technique, which turns out to be unnatural in our framework. 
Indeed we should require the sub/supersolution to belong to the  space 
$$
L^{\mu} (0,T; W^{1,\mu}_0 (\Omega)) \quad  \mbox{ with } \quad \mu= (\sigma-1)(p-q)+q
=N(q-(p-1))+2q-p
$$ (see   \cref{lemgrad}) and $\mu >p$ if  $p-\frac{N}{N+2}< q$. Consequently we would have such a result under the additional assumption that $u \in L^{\mu} (0,T; W^{1,\mu}_0 (\Omega))$.

\item[ii)]  Note that   the critical growth $q=p-\frac{N}{N+1}$, that corresponds to the case $\si=1$ and $m=r=1$,  has been excluded in \eqref{qu} (\eqref{q1prime} if $p=2$). For such a value we have  a slightly different result whose proof follows from the one of Theorem \ref{teo33},  with the following hypotheses on the data: 
$u_0\in L^{1+\om}(\Omega)$ and  $f\in L^{1+\om}(Q_T)$ with $\om>0$. %\footnote{riscrivere}
\end{itemize}
\end{remark}

\subsection{Assumptions for $2< p<N$}\label{hppge2}

In this case we change \eqref{p2} into the following problem: 
\begin{equation}\label{P}%%\tag{$\text{P}_C$}
\begin{cases}
\begin{array}{ll}
\ds  u_t-\text{div } \big(A(x)|\N u|^{p-2}\N u \big)=H(t,x,\N u)\qquad&\ds\text{in}\quad Q_T,\\
\ds  u (t,x) =0 & \ds\text{on}\quad(0,T)\times \partial \Omega,\\
 \ds u(0,x)=u_0(x) &\ds\text{in}\quad \Omega,
\end{array}
\end{cases}
\end{equation}
where the matrix $A(x)$ is bounded, coercive and with measurable coefficients, while the right hand side  satisfies a superlinear growth condition with respect to  the gradient.
 
More precisely, we assume that $A:\Omega\to\mathbb{R}^{N\times N}$ is a bounded and uniformly elliptic matrix with measurable coefficients, i.e. 
\begin{equation}\label{A}%%\tag{$\t{A}^C$}
\begin{array}{c}
A(x)=\{a_{ij}(x)\}_{i,j=1}^N \quad\text{with} \quad a_{ij}\in L^{\infty}(\Omega) \quad\forall  1\le i,j\le N\,,
\\[4mm] \dys 
\mbox{ such that } \quad \dys \exists\al,\,\b:\,\,  0<\al\le \beta \quad\text{and} \quad \al |\xi|^2\le A(x)\xi\cdot\xi\le \beta |\xi|^2 \quad\text{a.e.}\,\,x\in \Omega,\,\,
\forall \xi\in \mathbb{R}^N. 
\end{array}
\end{equation}
As far as the Hamiltonian term is concerned, we suppose, in addition to \eqref{H1},   that  $\exists M>0$, such that for any $\eps \in (0,1)$
\begin{equation}\label{H2}%%\tag{$\text{H}_C$}
\begin{array}{c}
\dys
H(t,x,\xi)- (1-\varepsilon)^{p-1} H\bigg((1-\varepsilon)^{p-2} \ t,x,\frac{\eta}{1-\varepsilon}\bigg) 
\leq c_2 \varepsilon^{1-q} \big| \xi - \eta \big|^q+  L|\xi-\eta| \left(   |\xi|^{\frac{p-2}2} + |\eta|^{\frac{p-2}2}\right) + \varepsilon M \\
[4mm]\ds
\text{with}\quad p-1<q<p,
\end{array}
\end{equation}
a.e. $(t,x)\in Q_T$, for all $\xi,\,\eta\in\R^N$ with $\xi\ne\eta$ and $\eps\in (0,1)$.\\
  
 The above hypothesis seems to be quite technical, since it combines several properties of the nonlinear lower order term. 
It is not so hard to see that, for example, the model Hamiltonian  
$$
H(t,x,\xi) = |\xi|^q + f_0 (t,x) 
$$ 
satisfies \eqref{H2} for $p-1<q<p$ and for a function $f_0$ bounded above and not increasing with respect to the $t$ variable. Let us underline that also some perturbations, through locally Lipschitz function, weighted with the $(p-2)/2$ power of the gradient of such a model Hamiltonian still fulfill hypothesis \eqref{H2}.

\medskip

As for the previous results, we have two regimes depending on the values of $q$. 

\medskip
We first deal with the range
 \begin{equation}\label{Q1}%%\tag{$\text{Q}_{1}^C$}
p-\frac{N}{N+1}< q<p
\end{equation}
that let us considering solutions in the class \eqref{cont}--\eqref{pot}.

\begin{theorem}\label{teoen1}
Assume \eqref{A}  and that $H(t,x,\xi)$ satisfies  \eqref{H1} together with  \eqref{F1}, \eqref{H2} and \eqref{Q1}. Let $u$ and $v$ be a subsolution and a supersolution of \eqref{P}, respectively, satisfying \eqref{cont} and \eqref{pot}. Assume that   $u_0\leq  v_0\le\bar{v}_0<\infty$ with $u_0$, $v_0\in L^\si(\Omega)$. 
Then  we have that $u\le v$ in ${Q_T}$.
\end{theorem}

Next, we consider the last case, that is the range 
\begin{equation}\label{Q3} %%\tag{$\text{Q}_{2}^C$}
p-1<q< p-\frac{N}{N+1}
\end{equation} 
and we have the following result. 

\begin{theorem}\label{teorinL1}
Assume \eqref{A}  and that $H(t,x,\xi)$ satisfies  \eqref{H1} together with  \eqref{F2}, \eqref{H2} and \eqref{Q3}. Let $u$ and $v$ be a subsolution and a supersolution of \eqref{P}, respectively, such that   \eqref{ET} holds. Assume that   $u_0\leq  v_0\le\bar{v}_0<\infty$ with $u_0$, $v_0\in L^\si(\Omega)$. 
Then  we have that $u\le v$ in ${Q_T}$.
\end{theorem}

\section{Notation and basic tools}\label{tool}
\setcounter{equation}{0}
\renewcommand{\theequation}{\thesection.\arabic{equation}}
\numberwithin{equation}{section}

With the purpose of  dealing with the bounded part of the sub/supersolutions considered during the paper, we here introduce a smooth approximation of the classical truncation function $T_n(s)=\max\{-n,\min\{n,s \}\}$. We define the smoothed truncation function $S_n(\cdot)$ and $\theta_n(\cdot)$ as follows:
\begin{equation}\label{ttn}
S_n(z)=\int_0^z  \theta_n(\tau)  \,d \tau,\qquad \mbox{while }\qquad
\theta_n(z)=
\begin{cases} 
\begin{array}{lrl}
1  & & |z|\le n,\\
\dys \frac{2n-|z|}{n}\qquad &n<&|z |\le 2n,\\
 0  & &   |z|>2n.
\end{array}
\end{cases}
\end{equation}

Moreover, here and in all the paper,  we denote by $G_k (z)$ the function
$$
G_k (z) = (z-k)_+ 
$$
for every $z\in\R $ and for any $k\geq 0$.

\medskip

Here we  recall a classical parabolic regularity result that we use systematically in the following. 

\begin{theorem}[Gagliardo-Nirenberg]\label{teoGN}
Let $\Omega\subset \mathbb{R}^N$ be a bounded and open subset and $T>0$; if
\begin{equation*}%\label{GN}
w\in L^{\infty}(0,T;L^{h}(\Omega))\cap L^{\eta}(0,T;W_0^{1,\eta}(\Omega))
\end{equation*}
with
\[
h,\, \eta\ge 1,\,\, \eta<N  \quad \text{and}\quad h\le\eta^* \,,
\]
then we have that
\begin{equation*}
w\in L^y(0,T;L^j(\Omega))
\end{equation*}
where the pair $(j,y)$ fulfils
\begin{equation*}%\label{relGN}
h\le j\le \eta^*,\,\,\eta\le y\le \infty
\qquad \mbox{ and } \qquad 
\frac{Nh}{j}+\frac{ N(\eta-h)+h\eta }{y}=N.
\end{equation*}
Moreover there exists a positive constant $c(N,\eta,h)$ such that the following inequality holds true:
\begin{equation}\label{disGN}\dys 
\int_0^T \|w(t)\|_{L^j(\Omega)}^y\,dt\le c(N,\eta,h)\|w\|_{L^{\infty}(0,T;L^h(\Omega))}^{y-\eta}\int_0^T\|\N w(t)\|_{L^{\eta}(\Omega)}^{\eta}\,dt.
\end{equation}
\end{theorem}

Next we state two useful Lemmata that we will use in the sequel in order to conclude the proofs of our results. 

\begin{lemma}\label{lemell1}
Let 
$
w\in L^\ro(0,T;W^{1,\ro}_0(\Omega))\cap L^\infty(0,T;L^\nu(\Omega)),$ with $\ro,\,\nu\ge 1
$   satisfy 
\begin{equation*}%\label{XX}
\exists c_0>0, \ \eta\ge 1 \ \mbox{ and } \ m>0 \, : \quad 
\|G_k(w)\|_{X}\le c_0\|G_k(w)\|_{X}^\eta\left(\iint_{E_k}B\,dx\,dt\right)^m \quad \mbox{for some } \quad B\in L^1(Q_T), 
\end{equation*}
for any $k\geq k_0$,  where $X$ is 
a Banach space, 
and 
$
E_k=\left\{
(t,x)\in Q_T:\,  w>k\,, \,|\N w|>0
\right\}.
$
Then $w\le 0$ in $Q_T$.
\end{lemma}
\begin{proof} The proof follows direclty from the one of
Lemma 2.1 in  \cite{Po}. 
\end{proof}
The next Lemma is a sort of parabolic version of the above one. 

\begin{lemma}\label{lemell}
Let 
$
w\in L^p(0,T;W^{1,p}_0(\Omega))\cap C([0,T];L^\si(\Omega))
$
 be a function satisfying the following inequality:
\begin{equation}\label{X}
\sup_{s\in (0,t)}\|w(s)\|_{L^\si(\Omega)}^\si+\|w\|_{L^p(0,t;W^{1,p}_0(\Omega))}^p\le c_0\sup_{s\in (0,t)}\|w(s)\|_{L^\si(\Omega)}^\eta \|w\|_{L^p(0,t;W^{1,p}_0(\Omega))}^{m}\qquad 
\end{equation}
with $0\le t\le T$ and for some $c_0>0$, $\eta>0$, $m\ge p$ and with $\|w(0)\|_{L^\si(\Omega)}=0$. 
Then $w\equiv0$ in $Q_T$.
\end{lemma}
\begin{proof}
Without loss of generality we take into account the case $m=p$. Then, since $\|w(0)\|_{L^\si(\Omega)}=0$ and by the continuity assumption we  define 
\[
T^*=\sup\left\{\tau>0: \quad \|w(s)\|_{L^\si(\Omega)}^\si\le\frac{1}{(2c_0)^\frac{\si}{\eta}}\quad \forall s\le \tau  \right\}>0
\]
which implies that, at least for $s\le T^*$, we get
\begin{equation}\label{dis}
\sup_{s\in (0,t)}\|w(s)\|_{L^\si(\Omega)}^\si+\frac{1}{2c_0}\|w\|_{L^p(0,t;W^{1,p}_0(\Omega))}^p\le 0
\end{equation}
from \eqref{X}. Now, let us suppose by contradiction that $T^*<T$. Then, if $t=T^*$ and by definition of $T^*$ we would find $\|w(T^*)\|_{L^\si(\Omega)}^\si=\frac{1}{(2c_0)^\frac{\si}{\eta}}\le 0$ 
which is in contrast with the assumption $c_0>0$. 
We thus deduce that \eqref{dis} holds for all $t\le T$ 
and, in particular, we conclude that $\|w(t)\|_{L^\ga(\Omega)}\equiv 0$ for every $t\in [0,T]$.
\end{proof}

\medskip

During the proof of our main results, we need some   regularity results, as the next two Lemmata.

\begin{lemma}\label{lemgrad}
%Assume \eqref{a2}--\eqref{a3} and 
Let   $u\in C([0,T];L^{\si}(\Omega))$ be such that  \eqref{pot} holds true. Then
\begin{equation}\label{gradreg}
|\N u|\in L^{N(q-(p-1))+2q-p}(Q_T)\quad\text{ with }\,\ \,1\leq \si\le 2 \quad \bigg(\mbox{i.e. } \frac{N}{N+2}\leq p-q\leq \frac{N}{N+1}\bigg)
\end{equation}
and
\begin{equation}\label{on}
\frac{1}{n}\iint_{\{n<|u|<2n\}} |\N u|^p 
%a(t,x,\N u)\N u
\,dx\,dt=o(n^{-(\si-1)})\qquad \text{as}\,\,n\to\infty.
\end{equation}
\end{lemma}
\begin{proof}
Let us start with the proof of the regularity in \eqref{gradreg}; using \eqref{pot} and \cref{teoGN}   with $\eta=p$ and $h=\frac{\si}{\ga}$ we deduce that $\ds u\in L^{p\frac{N\ga+\si}{N}}(Q_T) $. Then, by  Young's inequality,  we get 
\[
|\N u|^{p\frac{N\ga+\si}{N+\si}}\le c\frac{|\N u|^p}{(1+u)^{p(\ga-1)}}+c(1+u)^{p\frac{N\ga+\si}{N}}
\]
and \eqref{gradreg} follows by definitions of $\si$ and $\gamma$.

As far as \eqref{on} is concerned, we have the inequality
\[
\frac{n^{\si-1}}{n}\iint_{\{n<|u|<2n\}} |\N u|^p 
%a(t,x,\N u)\N u
 \,dx\,dt\le c_\gamma \iint_{\{n<|u|<2n\}}
 %\bigg[ 
 |\N |u|^\ga |^p
 %+c\frac{\ell (t,x)}{n^{2-\si}} \bigg]
 \,dx\,dt\,.
\]
%holds true thanks to \eqref{a3} with  $\ell\in L^{p'}(Q_T)$. 
Then, since $n\geq1$ and consequently dealing with $|u|>1$, it implies that we can employ \eqref{pot} and thus $\meas\{(t,x)\in(0,T)\times \Omega \,:\,   n<|u|\le 2n\}\to 0$ as $n\to \infty$ and thus   \eqref{on} follows.
\end{proof}

\begin{lemma}\label{lemmaq}
Let $u,\, v$ be respectively a renormalized subsolution and a renormalized supersolution of \eqref{Dp}   such that \eqref{ET}, \eqref{H1} with \eqref{q3} hold. Then
\begin{equation}\label{gradpos}
|\N u_+| ,\,|\N v_-| \in L^r(Q_T)\quad\mbox{for } \quad   1\leq r<p-\frac{N}{N+1}.
\end{equation}
\end{lemma}
\begin{proof}
We only deal with the case of subsolution $u$, since having $v$ be a supersolution implies that $-v$ is a subsolution.

In fact, \eqref{gradpos}   is a consequence of Corollary 4.7 (see also Remark 4.2) applied to the positive part of $u$, that yields to the following inequality

%First, we observe that $T_k(w)=T_k(S_n(w))$ with $S_n(w)=\int_0^w\te_n(z)\,dz$ for $n$ large enough.\\
%We multiply the equation in \eqref{Dp} by $T_k (\up) $ and integrate over $Q_T$, getting
%\begin{equation*}
%\begin{array}{c}
%\ds
%\integrale \Theta(S_n( u(T) ))\,dx +\alpha \iint_{Q_T}|\N T_k (\up) |^p\,dx\,dt\\
%[4mm]
%\ds\le k\left[\iint_{Q_T}|\N u|^r\,dx\,dt+\|f\|_{L^1(Q_T)}+\|u_0\|_{L^1(\Omega)} \right]+\frac{ck}{n}  \iint_{\{n<|u|<2n\}} |a(t,x,\N u)||\N u|\,dx\,ds
%\end{array}
%\end{equation*}
%thanks to \eqref{a3}, \eqref{H1} and where $\Theta(w)=\int_0^wT_k(z)\,dz$. The asymptotic energy condition allows us to let $n\to\infty$ obtaining that
\begin{equation*}
\begin{array}{c}
\ds
\integrale \Theta( u_+ (T) )\,dx +\alpha \iint_{Q_T}|\N T_k (\up) |^p\,dx\,dt
\ds \le k \ds \left[\|H(t,x, \N u) \|_{L^1(Q_T)}+\|u_0\|_{L^1(\Omega)} \right]\,.
\end{array}
\end{equation*}
By the standard results on the regularity, we deduce that $\up\in L^r(0,T;W^{1,r}(\Omega))$ for every $r<p-\frac{N}{N+1}$  (see   \cite{ST} and \cite{BDGO}).
\end{proof}

We conclude this Section with another useful result.

\begin{lemma}\label{leml1}
Let $\ro,\,m\ge 1$ and $w\in \mathcal{T}^{1,\ro}_0(Q_T)\cap L^\infty(0,T;L^m(\Omega)) $ satisfy
\begin{equation}\label{gr1}
\|w\|_{L^\infty(0,T;L^m(\Omega))}^m\le L
\qquad 
\mbox{ and } \qquad 
%\begin{equation}\label{gr2}
\iint_{Q_T}\frac{|\N w|^\ro}{(|w|+\mu)^\ga}\,dx\,dt\le M\mu^{-\nu}
\end{equation}
where $L, 	\ M,\,\mu,\,\nu>0$ and $0<\ga<\ro$. Then  
\begin{equation}\label{gradb}
\exists c=c(\rho, \gamma, N,\mu,m) \,:\qquad \||\N w|^{b}\|_{L^1(Q_T)}\le c \left(L^{\frac{\ga-\nu}{N+m}}M\right)^{\frac{b(N+m)}{N(\ro-\ga+\nu)+m\ro}} \quad\mbox{with}\quad b=\frac{N(\ro-\ga)+m\ro}{N+m}.
\end{equation}
\end{lemma}
\begin{proof}
The above assumptions on $w$ imply that we can apply \cref{teoGN} to the function
\[
\left((|T_k(w)|+\mu)^{\frac{\ro-\ga}{\ro}}-\mu^{\frac{\ro-\ga}{\ro}}\right)\in 
L^\infty(0,T;L^{\frac{m\ro}{\ro-\ga}}(\Omega))\cap L^\ro(0,T;W^{1,\ro}_0(\Omega))
\]
and we get the regularity estimate
\begin{equation}\label{regk}
\begin{array}{c}
\ds
\iint_{Q_T}\left((|T_k(w)|+\mu)^{\frac{\ro-\ga}{\ro}}-\mu^{\frac{\ro-\ga}{\ro}}\right)^{\ro\frac{N+\frac{m\ro}{\ro-\ga}}{N}}\,dx\,dt\\
[4mm]\ds
\le \bigg\|(|T_k(w)|+\mu)^{\frac{\ro-\ga}{\ro}}-\mu^{\frac{\ro-\ga}{\ro}}\bigg\|_{L^\infty(0,T;L^{\frac{m\ro}{\ro-\ga}}(\Omega))}^{\frac{m\ro^2}{N(\ro-\ga)}}
\iint_{Q_T}\frac{|\N T_k(w)|^\ro}{(|T_k(w)|+\mu)^\ga}\,dx\,dt.
\end{array}
\end{equation}
Then, since the inequality $(a+b)^\al \leq  a^\al+b^\al$ holds for $a,\,b>0$,  $0<\al<1$ and   $\frac{\ro-\ga}{\ro}<1$, we have by \eqref{regk},  combined with \eqref{gr1},   and Fatou's Lemma   
\begin{equation*}
\begin{array}{c}
\ds
\iint_{Q_T}\left((|w|+\mu)^{\frac{\ro-\ga}{\ro}}\right)^{\ro\frac{N+\frac{m\ro}{\ro-\ga}}{N}}\,dx\,dt
\le 
c\left(
\mu^{\frac{N(\ro-\ga)+m\ro}{N}}+L^{\frac{\ro}{N}}M\mu^{-\nu}
\right).
\end{array}
\end{equation*}
Minimizing the right hand side above with respect to  $\mu$,  we get that the minimum is achieved at 
$\mu=c\left(L^{\frac{\ro}{N}}M\right)^{\frac{N}{N(\ro-\ga+\nu)+m\ro}}$
for a constant $c$ that depends only on  $\ro,\,\ga,\,N,\,\nu,\,m$. \\
We are now ready to prove \eqref{gradb}: we use the H\"older's inequality with $(\frac{\ro}{b},\frac{\ro}{\ro-b})$ in order to get 
\begin{equation*}
\iint_{Q_T}|\N w|^b\,dx\,dt\le \left(\iint_{Q_T}\frac{|\N w|^\ro}{(|w|+\mu)^\ga}\,dx\,dt\right)^{\frac{b}{\ro}}\left(\iint_{Q_T}(|w|+\mu)^{\frac{\ga b}{\ro-b}}\,dx\,dt\right)^{\frac{\ro-b}{\ro}}
\end{equation*}
and we use that  $b$ satisfies $\frac{\ga b}{\ro-b}= \frac{N(\ro-\ga)+m\ro}{N}$, so that   \eqref{gradb} follows from the choice of $\mu$.
\end{proof}

\section{Proofs in the case    $1<p\le 2$}
\sectionmark{The case with $p=2$}

\setcounter{equation}{0}
\renewcommand{\theequation}{\thesection.\arabic{equation}}

\numberwithin{equation}{section}

We start by proving   \cref{teorin2} and   \cref{thm25}. 

\begin{proof}[Proof of \cref{teorin2} and of  \cref{thm25}]
As anticipated, we need to rewrite the inequalities satisfied by sub/supersolutions in terms of their bounded parts plus some (quantified) reminder. To this aim, we set
 \begin{equation}\label{unvn}
\begin{array}{c}
\ds u_n=S_n(u)\quad \text{and}\quad v_n=S_n(v)
\end{array}
\end{equation}
where $S_n(\cdot)$ has been defined in \eqref{ttn}.
We consider the renormalized formulation in \eqref{renf} with $S(u)=\theta_n(u)$ so we obtain
\begin{equation*}
\begin{array}{c}
\ds
\int_\Omega u_n(t)\varphi(t)\,dx+\iint_{Q_t} a(s,x,\N u)\tenu\cdot \N\vp+a(s,x,\N u)\cdot \N u \,\te'_n(u)\vp\,dx\,ds\\
[4mm]\ds
\le \iint_{Q_t} H(s,x,\N u)\tenu\varphi\,dx\,ds+\int_\Omega u_n(0)\varphi(0)\,dx.
\end{array}
\end{equation*}
Reasoning in the same way on the supersolution $v$ - of course, with $S(v) = \tenv $ - and considering the difference between the above inequalities, we get $\forall 0<t\leq T$,  
\begin{equation*}
\begin{array}{c}
\ds
\int_\Omega (u_n(t)-v_n(t))\varphi(t)\,dx\\
[4mm]\ds+\iint_{Q_t} \left[a(s,x,\N u)\tenu-a(s,x,\N v)\tenv\right]\cdot \N\vp+\left[a(s,x,\N u)\cdot \N u\,\te'_n(u)-a(s,x,\N v)\cdot \N v\,\te'_n(v)\right]\,\vp\,dx\,ds\\
[4mm]\ds
\le \iint_{Q_t}\left[H(s,x,\N u)\tenu-H(s,x,\N v)\tenv\right]\varphi\,dx\,ds+\int_\Omega [u_n(0)-v_n(0)]\varphi(0)\,dx.
\end{array}
\end{equation*}
We now define
\begin{equation}\label{zn}
\zn=\un-\vn
\end{equation}
and rewrite the above inequality as
\begin{equation}\label{zetaeps2}
\begin{array}{c}
\ds
\int_\Omega \zn(t)\varphi(t)\,dx+\iint_{Q_t} \left[a(s,x,\N \un)-a(s,x,\N \vn)\right]\cdot \N\vp\,dx\,ds\\
[4mm]\ds
\le \iint_{Q_t} \left[H(s,x,\N \un)-H(s,x,\N \vn)\right]\varphi\,dx\,ds
+\Rn
\end{array}
\end{equation}
where
\begin{equation*}
\begin{split}
\ds\Rn&=\iint_{Q_t}\left[ a(s,x,\N \un)- a(s,x,\N \vn)
-a(s,x,\N u)\tenu+a(s,x,\N v)\tenv\right]\cdot \N\vp
\,dx\,ds\\
&\quad-\iint_{Q_t} \left[a(s,x,\N u)\cdot \N u\,\te'_n(u)-a(s,x,\N v)\cdot \N v\,\te'_n(v)\right]\,\vp\,dx\,ds
\\
&\quad
+\iint_{Q_t} \left[H(s,x,\N u)\tenu-H(s,x,\N v)\tenv- H(s,x,\N \un)+H(s,x,\N \vn)\right]\varphi\,dx\,ds\\
&\quad+\int_\Omega \zn(0)\varphi(0)\,dx.
\end{split}
\end{equation*}
 In virtue of the density result  \cite[Proposition $4.2$]{PPP}, we are allowed to take 
 %\footnote{qui diciamo qualcosa, citiamo [PPP]?giusto...manca la prop. 4.2, no? Ho aggiunto la referenza ma non so se servono due parole in più. l'ho messa anche nelle dimostrazioni successive}
\[
\vp(\zn)=\left[(G_k(\zn)+\mu)^{\si-1}\right]-\mu^{\si-1},\quad \mu>0,
\]
in the inequality in \eqref{zetaeps2} and then, thanks also to \eqref{a2} and \eqref{H3}, we obtain
\begin{equation} \label{inn2}
\begin{array}{c}
\ds
\integrale \Vp_k (\zn(t))\,dx
+\al(\si-1)\iint_{Q_t}\left||\N \un|^{2}+|\N \vn|^{2}\right|^{\frac{p-2}{2}}\frac{|\N G_k(\zn)|^2}{(G_k(\zn)+\mu)^{2-\si}}\,dx\,ds\\
[4mm]\ds
\le c_1 \iint_{Q_t} |\N G_k(\zn )|\left[
g+|\N \un|^{q-1}+|\N \vn|^{q-1}
\right]\left(\left[(G_k(\zn)+\mu)^{\si-1}\right]-\mu^{\si-1}\right)\,dx\,ds +\Rn
\end{array}
\end{equation}
where $  \Vp_k (z )\,dx=\int_0^{G_k(z)}\vp(\tau)\,d\tau$.

\medskip 

We want to prove now  that $\dys \lim_{n\to\infty}\Rn=0$. From now on, we denote by  $\om_n $ any quantity that vanishes as $n$ diverges, and we  set
\begin{equation}\label{omn}
\Rn=I_1+I_2+I_3+I_4
\end{equation}
where
\begin{equation}\label{I}
\begin{array}{lll}
\ds I_1&\ds =\iint_{Q_t}\big[\left[a(s,x,\N \un)- a(s,x,\N \vn)\right] -
\left[\tenu a(s,x,\N u)-\tenv a(s,x,\N v)\right]\big]\cdot \N \zn\vp'(\zn)\,dx\,ds,&\\
[4mm]
\ds I_2&\ds =-\iint_{Q_t} \left[\te'_n(u) a(s,x,\N u)\cdot \N u-\te'_n(v)a(s,x,\N v)\cdot \N v\right]\vp(\zn)\,dx\,ds,&\\
[4mm]
\ds I_3&\ds =\iint_{Q_t} \big[\left[\tenu H(s,x,\N  u)-\tenv H(s,x,\N v)\right]-\left[H(s,x,\N \un)-H(s,x,\N \vn)\right]\big]\vp(\zn)\,dx\,ds&\\
\ds I_4&\ds =\int_\Omega \zn(0)\varphi(0)\,dx.&
\end{array}
\end{equation}
Let us start by studying   $I_1$. The definition of $\te_n(\cdot)$ and \eqref{a4} imply that
\begin{equation}\label{I1}
\begin{split}
I_1&=\iint_{\{n<|u|<2n\}} \big[a(s,x,\N \un)-\tenu a(s,x,\N u)\big]\cdot \N (u_n-v_n)\vp'(u_n-v_n)\,dx\,ds\\
&\quad-\iint_{\{n<|v|<2n\}}\big[a(s,x,\N \vn))-\tenv a(s,x,\N v)\big]\cdot \N (u_n-v_n)\vp'(u_n-v_n)\,dx\,ds
\end{split}
\end{equation}
being $a(s,x,\N u)\theta_n(u)-a(s,x,\N u_n)\equiv0$ when $|u|\le n$ and since $\N u_n  \equiv 0$ when $|u|\ge 2n$ (and the same hols for $v_n$). We just prove that the first integral in \eqref{I1} behaves as $\omn$ since the second one can be dealt in the same way. The definition of $\vp(\cdot)$ and \eqref{a3} allow us to deduce that
\begin{equation}\label{I11}
\begin{array}{c}
\ds
\iint_{\{n<|u|<2n\}}\big[a(s,x,\N \un)-\tenu a(s,x,\N u)\big]\cdot \N (u_n-v_n)\vp'(u_n-v_n)\,dx\,ds\\
[4mm]\ds
\le cn^{-(2-\si)}\bigg[
\iint_{\{n<|u|<2n\}}|\N u|^p\,dx\,ds+\iint_{\{n<|u|<2n\}}|\N u|\ell\,dx\,ds\\
[4mm]\ds+
\iint_{\{n<|u|<2n,\,|v|<2n\}}|\N u|^{p-1}|\N v|\,dx\,ds+\iint_{\{n<|v|<2n\}}|\N v|\ell\,dx\,ds
\bigg]\,,
\end{array}
\end{equation}
and thanks to  \eqref{pot} we estimate the first integral in the right hand side above, since 
\[
n^{-(2-\si)}\iint_{\{n<|u|<2n\}}|\N u|^p\,dx\,ds\le c\iint_{\{n<|u|<2n\}}|\N |u|^\ga|^p\,dx\,ds=\om_n \,,
\]
using that  $|\{(t,x)\in(0,T)\times \Omega \,:\,   n<|u|<2n\}|\to 0$ as $n\to \infty$.
As far as the third integral is concerned, H\"older's inequality with $(p,p')$ implies
\[
\begin{array}{c}
\ds
n^{-(2-\si)}\iint_{\{n<|u|<2n,\,|v|<2n\}}|\N u|^{p-1}|\N v|\,dx\,ds\\
[4mm]\ds
\le cn^{-(2-\si)}\left(\iint_{\{n<|u|<2n\}}|\N u|^p\,dx\,ds\right)^{\frac{1}{p'}}\left(
\iint_{\{|v|<2n\}}|\N v|^p\,dx\,ds
\right)^{\frac{1}{p}}\\
[4mm]\ds
\le \left(\iint_{\{n<|u|<2n\}}|\N |u|^\ga|^p\,dx\,ds\right)^{\frac{1}{p'}}\left(
\iint_{Q_T }|\N |v|^\ga|^p\,dx\,ds
\right)^{\frac{1}{p}}
=\om_n \,,
\end{array}
\]
thanks again to \eqref{pot}. Finally, we deal with the second term in \eqref{I11} (the fourth is treated in the same way) applying again H\"older's inequality with $(p,p')$ and so obtaining
\[
\begin{array}{c}
\ds cn^{-(2-\si)}\iint_{\{n<|u|<2n\}}|\N u|\ell\,dx\,ds \\
[4mm]\ds
\le cn^{-(2-\si)}\left(\iint_{\{n<|u|<2n\}}|\N u|^p\,dx\,ds\right)^{\frac{1}{p}}\left(
\iint_{\{n<|u|<2n\}}|\ell|^{p'}\,dx\,ds
\right)^{\frac{1}{p'}}\\
[4mm]\ds
\le \left(\iint_{\{n<|u|<2n\}}|\N |u|^\ga|^p\,dx\,ds\right)^{\frac{1}{p}}\left(n^{-(2-\si)}
\iint_{Q_T}|\ell|^{p'}\,dx\,ds
\right)^{\frac{1}{p'}} =\om_n .
\end{array}
\]
A similar argument can be done for the last term in \eqref{I1}, so that  we have that $I_1=\om_n $.\\

By the definition of $\te_n(\cdot)$ and since $|\vp(u_n-v_n)|\le cn^{\si-1}$,   $I_2$   can be estimated as 
\[
I_2\le \frac{c}{n^{2-\si}}\left[\iint_{\{n<|u|<2n\}} |a(s,x,\N u)|| \N u|\,dx\,ds+
\iint_{\{n<|v|<2n\}}|a(s,x,\N v)|| \N v|\,dx\,ds\right]\,,
\]
and $I_2=\om_n $ thanks to \eqref{a3} and \cref{lemgrad} (see \eqref{on}).\\

As far as   $I_3$   is concerned, we have that
\begin{align*}
|I_3|&\le cn^{\si-1}\biggl[\iint_{\{n<|u|<2n\}} \left|\tenu H(s,x,\N  u)-H(s,x,\N \un)\right|\,dx\,ds\\
&\quad +\iint_{\{n<|v|<2n\}} \left|\tenv H(s,x,\N  v)-H(s,x,\N \vn)\right|\,dx\,ds\biggr]
\end{align*}
by definition of $\te_n(\cdot)$: indeed, $\tenu H(s,x,\N  u)-H(s,x,\N \un)=0$ if $|u|\le n$ and $|\N \un|=0$ if $|u|\ge 2n$. We only consider the first term in the inequality above since the second one can be dealt with in the same way. Thanks to the growth assumption \eqref{H111}, the desired convergence of the first term follows once we prove that
\[
\begin{array}{c}\dys 
cn^{\si-1}\iint_{\{n<|u|<2n\}} \left|\tenu H(s,x,\N  u)-H(s,x,\N \un)\right|\,dx\,ds\\
[4mm]\ds
\le c\left[ \iint_{\{n<|u|<2n\}}|\N u|^q|u|^{\sigma-1}\,dx\,ds+\iint_{\{n<|u|<2n\}}|f||u|^{\sigma-1}\,dx\,ds\right] =\omn.
\end{array}
\]
An application of H\"older's inequality with indices $\left(\frac{p}{q},\frac{p^*}{p-q},\frac{N}{p-q}\right)$, Sobolev's embedding and \cref{teoGN} (see \eqref{disGN}) lead us to
\begin{equation}\label{grad}
\begin{split}
\iint_{\{n<|u|<2n\}}|\N u|^q|u|^{\sigma-1}\,dx\,ds&\le  \int_0^t \left(\int_{\{n<|u(s)|\le 2n \}}|\N |u(s)|^\ga |^p\,dx\right)^{\frac{q}{p}}\left(\int_{\{n<|u(s)|\le 2n \}}|u(s)|^{\si-1+\frac{q}{p-q}}\,dx\right)^{\frac{p-q}{p}}\,ds\\
&\le c\sup_{s\in [0,t]}\left(\integrale|u(s)|^{\si}\,dx\right)^{\frac{p-q}{N}}\iint_{\{n<|u|<2n \}}|\N |u|^\ga|^p\,dx\,ds=\om_n 
\end{split}
\end{equation}
for the same reasons given above. As far as the integral involving the forcing term is concerned, we observe that, applying the H\"older inequality with indices $(m,m')$ and $(r,r')$ as in \eqref{F1}, we get
\begin{align*}
\iint_{\{n<|u|<2n\}}|f||u|^{\sigma-1}\,dx\,ds\le 
\|f \|_{L^r(0,T;L^m(\Omega))}
\left[\int_0^t\left( \int_{\{n<|u(s)|\le 2n\}}  |u(s)|^{\ga\frac{m'(\sigma-1)}{\ga}}\,dx\right)^{\frac{r'}{m'}}\,ds\right]^{\frac{1}{r'}}.
\end{align*} 
We go further invoking \cref{teoGN} with $w=|u|^{\ga}$ and, in particular, the inequality in \eqref{disGN} becomes
\begin{equation}\label{frn}
\int_0^t\left( \int_{\{n<|u(s)|\le 2n\}} |u(s)|^{\ga m}\,dx\right)^{\frac{y}{m}}\,ds\le {c_{GN}}\|u\|_{L^{\infty}(0,T;L^{\si}(\Omega))}^{\ga(y-p)}
\iint_{\{n<|u|<2n\}}  |\N |u|^\ga|^p\,dx\,ds.
\end{equation}
We observe that since the couple $(m,r)$ satisfies \eqref{F1}, we have that the pair  $(j,y)$ fulfills
\[
j\ge m'\frac{\sigma-1}{\ga}\quad\text{and}\quad y\ge r'\frac{\sigma-1}{\ga}.
\]
We thus proceed through Lebesgue spaces inclusion and we deduce 
\begin{equation*}
\begin{array}{c}
\ds
\| f \|_{L^r(0,T;L^m(\Omega))}
\left[\int_0^t\left( \int_{\{n<|u(s)|\le 2n\}} |u(s)|^{\ga w}\,dx\right)^{\frac{y}{w}}\,ds\right]^{\frac{\si-1}{\ga y}}\\
[4mm]\ds
\le c\|f \|_{L^r(0,T;L^m(\Omega))}\left(
\|u\|_{L^{\infty}(0,T;L^{\si}(\Omega))}^{\ga(y-p)}
\iint_{\{n<|u|<2n\}}  |\N |u|^\ga|^p\,dx\,ds
\right)^{\frac{\si-1}{\ga y}}=\omn, 
\end{array}
\end{equation*}
that implies $I_3= \omn.$ Finally,  recalling that $u_0\le v_0$ and the definition of $\te_n(\cdot)$, we conclude that also $I_4 \leq  \omn.$ so that that $\Rn=\omn$.

%The main ingredients we used in this step are the uniqueness class \eqref{pot} and the structure of $\vp(\cdot)$.\footnote{sta frase non si capisce... togliere? \\ modificata}\\

\medskip

We now get into the main step of the proof. Let  $E_{n,k}$ be the subset of $Q_t$ defined by 
\begin{equation}\label{En1}
E_{n,k}=\left\{ (t,x)\in Q_t:\,\, z_n>k\quad\text{and}\quad |\N z_n|>0  \right\},
\end{equation}
and we set 
\begin{equation}\label{BBn}
\begin{array}{c}
\ds
 \B_{n,1}=
\left[g+
|\N \un|^{q-1}+|\N \vn|^{q-1}\right]\kin
,\qquad 
 \B_{n,2}=
\left[
|\N \un|^{2}+|\N \vn|^{2}
\right]^{\frac{a(2-p)}{2(2-a)}}\kin\,,
\end{array}
\end{equation}
where the parameter $a\le p\le 2$ has to be fixed. 

%\noindent
%\textit{Uniform boundedness}\\
%\noindent
An application of H\"older's inequality with indices $(\frac{2}{a},\frac{2}{2-a})$ and the inequality in \eqref{inn2} provide with the following estimate:
\begin{equation}\label{s}
\begin{array}{c}
\ds
\iint_{Q_t}\frac{|\N G_k(\zn)|^a}{\left(G_k(\zn)+\mu\right)^{\frac{a}{2}(2-\si)}}
\,dx\,ds\\
[4mm]\ds
\le
\left(\iint_{E_{n,k}} \B_{n,2}\,dx\,ds
\right)^{\frac{2-a}{2}}
\left(
\iint_{E_{n,k}}\frac{|\N G_k(\zn)|^2}{\left(G_k(\zn)+\mu\right)^{2-\si}}
\bigg[
|\N {u}_n|^{2}+|\N {v}_n|^{2}
\bigg]^{\frac{p-2}{2}}\,dx\,ds
\right)^{\frac{a}{2}}
\\
[4mm]\ds
\le c\left(\iint_{E_{n,k}} \B_{n,2}\,dx\,ds
\right)^{\frac{2-a}{2}} \Biggl(
\iint_{Q_t}\left|\N G_k(\zn)\right|\B_{n,1}
\left[\left(G_k(\zn)+\mu\right)^{\si-1}-\mu^{\si-1}\right]\,dx\,ds+\omn
\Biggr)^{\frac{a}{2}}.
\end{array}
\end{equation}
We recall \eqref{h1} (i.e. we know that $g\in L^d(\Omega)$ for $d=\frac{N(q-(p-1))-p+2q}{q-1}$) and we set $a$ such that $\frac{a(2-p)}{2-a}=d(q-1)=N(q-(p-1))+2q-p$, namely
\[
a=\frac{2d(q-1)}{2-p+d(q-1)}=\frac{2\left(N(q-(p-1))-p+2q\right)}{(N+2)(q-(p-1))}.
\] 
Then, the gradient regularity \eqref{gradreg} contained in \cref{lemgrad} applied on both $|\N\un|$ and $|\N\vn|$ (we recall also \eqref{pot}) implies that the first integral in the right hand side  of \eqref{s} is finite.\\
Now, let us focus on the second one. H\"older's inequality with   indices $\left(a,d,\frac{2}{\si}\frac{p(N\ga+\si)}{N}\right)$ (indeed, $1-\frac{1}{a}-\frac{1}{d}=\frac1{2} \frac{\si(p-q)}{\si(p-q)-p+2q}=\frac{2}{\si}\frac{p(N\ga+\si)}{N}$ by definitions of $\si$ and $\ga$) yields to
\begin{equation}\label{q-1}
\begin{array}{c}
\ds
\iint_{Q_t}\left|\N G_k(\zn)\right|\B_{n,1}
\left[\left(G_k(\zn)+\mu\right)^{\si-1}-\mu^{\si-1}\right]\,dx\,ds\\
[4mm]\ds\le
\iint_{Q_t} \biggl[\left|\N G_k(\zn)\right|\left(G_k(\zn)+\mu\right)^{\frac{\si-2}{2}}\B_{n,1}\left[\left(G_k(\zn)+\mu\right)^{\si-1}-\mu^{\si-1}\right]^{\frac{1}{2}} \left(G_k(\zn)+\mu\right)^{\frac{1}{2}}\biggr]\,dx\,ds
\\
[4mm]\ds
\le \left(\iint_{Q_t}\frac{|\N G_k(\zn)|^a}{\left(G_k(\zn)+\mu\right)^{\frac{a}{2}(2-\si)}}\,dx\,ds\right)^{\frac{1}{a}}\left(\iint_{E_{n,k}}  \B_{n,1}^d \,dx\,ds
\right)^{\frac{1}{d}}\times\\  
[4mm]\ds
\times\left(\iint_{Q_t}
\left(\left[\left(G_k(\zn)+\mu\right)^{\si-1}-\mu^{\si-1}\right]^{\frac{1}{2}}\left(G_k(\zn)+\mu\right)^{\frac{1}{2}}\right)^{\frac{2p(N\ga+\si)}{N\si}}\,dx\,ds
\right)^{\frac{N\si}{2p(N\ga+\si)}}
\end{array}
\end{equation}
and then we take advantage of \eqref{q-1} in \eqref{s}, we obtain
\begin{equation}\label{gnn}
\begin{array}{c}
\ds
\iint_{Q_t}\frac{|\N G_k(\zn)|^a}{\left(G_k(\zn)+\mu\right)^{\frac{a}{2}(2-\si)}}\,dx\,ds\\
[4mm]\ds
\le c\left(\iint_{E_{n,k}}  \B_{n,1}^d \,dx\,ds
\right)^{\frac{a}{d}}\left(\iint_{E_{n,k}} \B_{n,2}\,dx\,ds
\right)^{2-a}\times
\\
[4mm]\ds
\times\left(\iint_{Q_t}
\left(\left[\left(G_k(\zn)+\mu\right)^{\si-1}-\mu^{\si-1}\right]^{\frac{1}{2}}\left(G_k(\zn)+\mu\right)^{\frac{1}{2}}\right)^{\frac{2p(N\ga+\si)}{N\si}}
\,dx\,ds\right)^{\frac{N\si}{2p(N\ga+\si)}a}+\omn%^a
.
\end{array}
\end{equation}
We observe  that   the first two integrals in the right hand side above are bounded thanks to \eqref{h1}, the definition of $a$ and \cref{lemgrad}. Moreover, using \eqref{gnn} in \eqref{q-1} leads to
\begin{equation}\label{q-2}
\begin{array}{c}
\ds
\iint_{Q_t}\left|\N G_k(\zn)\right|\B_{n,1}
\left[\left(G_k(\zn)+\mu\right)^{\si-1}-\mu^{\si-1}\right]\,dx\,ds\\
[4mm]\ds
 \le c\left(\iint_{E_{n,k}}  \B_{n,1}^d \,dx\,ds
\right)^{\frac{2}{d}}\left(\iint_{E_{n,k}} \B_{n,2}\,dx\,ds
\right)^{\frac{2-a}{a}}\times
\\
[4mm]\ds
\times\left(\iint_{Q_t}
\left(\left[\left(G_k(\zn)+\mu\right)^{\si-1}-\mu^{\si-1}\right]^{\frac{1}{2}}\left(G_k(\zn)+\mu\right)^{\frac{1}{2}}\right)^{\frac{2p(N\ga+\si)}{N\si}}
\,dx\,ds\right)^{\frac{N\si}{p(N\ga+\si)}}+\omn
\end{array}
\end{equation}
and the uniform estimate of the right hand side  in \eqref{inn2} is closed.

\medskip 

Next,  we observe that the definitions of $a$, $\ga$ and $\si$ imply that
\[
\frac{2}{\si}\frac{p(N\ga+\si)}{N}=a\frac{N+2}{N}
\]
where $a\frac{N+2}{N}$ is the Gagliardo-Nirenberg regularity exponent applied with spaces
 \[
 L^\infty(0,T;L^2(\Omega))\cap L^a(0,T;W^{1,a}_0(\Omega)).
\]
In particular, the inequality in \eqref{disGN} applied to the   the function $$\left[\left(G_k(\zn)+\mu\right)^{\si-1}-\mu^{\si-1}\right]^{\frac{1}{2}}\left(G_k(\zn)+\mu\right)^{\frac{1}{2}}$$
gives 
\begin{equation}\label{gnnn}
\begin{array}{c}
\ds
\iint_{Q_t}\left(\left[\left (G_k(\zn) +\mu\right)^{\si-1}-\mu^{\si-1}\right]^{\frac{1}{2}}\left( G_k(\zn) +\mu\right)^{\frac{1}{2}}\right)^{a\frac{N+2}{N}}\,dx\,ds\\
[4mm]\ds
\le 
{c_{GN}}\left\| G_k(\zn) +\mu\right\|_{L^\infty(0,t;L^\si(\Omega))}^{\frac{a\si}{N}} \ 
\iint_{Q_t}
\frac{|\N ( G_k(\zn) )|^a}{\left( G_k(\zn) +\mu\right)^{\frac{a}{2}(2-\si)}}\,dx\,ds.
\end{array}
\end{equation}
So far, we know that the second integral in \eqref{gnnn} is bounded thanks to \eqref{gnn}. Furthermore, it holds from \eqref{inn2} and \eqref{q-2} that
\begin{equation*}
\begin{array}{c}
\ds
\| G_k(\zn) +\mu\|_{L^\infty(0,t;L^\si(\Omega))}^\si\le
\integrale\Vp_k (z_n(t))\,dx\\
[4mm]\ds
 \le c\left(\iint_{E_{n,k}}  \B_{n,1}^d \,dx\,ds
\right)^{\frac{2}{d}}\left(\iint_{E_{n,k}} \B_{n,2}\,dx\,ds
\right)^{\frac{2-a}{a}} \times
\\
[4mm]\ds
\times \left(\iint_{Q_t}\left(\left[\left( G_k(\zn) +\mu\right)^{\si-1}-\mu^{\si-1}\right]^{\frac{1}{2}}\left( G_k(\zn) +\mu\right)^{\frac{1}{2}}\right)^{a\frac{N+2}{N}}
\,dx\,ds\right)^{\frac{2N}{a(N+2)}}+\omn
\end{array}
\end{equation*}
and then, using \eqref{gnn} and   \eqref{gnnn}, it follows that
\begin{equation*} %\label{gnnn}
\begin{array}{c}
\ds
\iint_{Q_t}\left(\left[\left( G_k(\zn) +\mu\right)^{\si-1}-\mu^{\si-1}\right]^{\frac{1}{2}}\left( G_k(\zn) +\mu\right)^{\frac{1}{2}}\right)^{a\frac{N+2}{N}}\,dx\,ds\\
[4mm]\ds
\le 
c\left(\iint_{E_{n,k}} \B_{n,1}^d \,dx\,ds
\right)^{\frac{2a}{dN}}\left(\iint_{E_{n,k}} \B_{n,2}\,dx\,ds
\right)^{\frac{2-a}{N}} \times
\\
[4mm]\ds
\times\left(\iint_{Q_t}
\left(\left[\left( G_k(\zn) +\mu\right)^{\si-1}-\mu^{\si-1}\right]^{\frac{1}{2}}\left( G_k(\zn) +\mu\right)^{\frac{1}{2}}\right)^{a\frac{N+2}{N}}
\,dx\,ds\right)
+\omn^{\frac{a}{N}}\, .
\end{array}
\end{equation*}
Finally, passing to the limit first with respect to $n$ (we apply Lebesgue Theorem on the right hand side)  and subsequently as $\mu\to 0$, we  obtain: 
\begin{equation*} %\label{gnnn}
\begin{array}{c}
\ds
\iint_{Q_t} G_k(u-v) ^{\frac{a\si}{2}\frac{N+2}{N}}\,dx\,ds
\le 
c\left(\iint_{E_k} \B_1^d\,dx\,ds
\right)^{\frac{2a}{dN}}\left(\iint_{E_k} \B_2\,dx\,ds
\right)^{\frac{2-a}{N}}\left(
\iint_{Q_t} G_k(u-v) ^{\frac{a\si}{2}\frac{N+2}{N}}\,dx\,ds\right)
\end{array}
\end{equation*}
where
\begin{equation}\label{EE}
E_k=\left\{
(t,x)\in Q_t:\quad u-v>k\,\,\text{and}\,\, |\N(u-v)|>0
\right\}
\end{equation}
while 
\begin{equation*} %\label{B1B2}
\begin{array}{c}
\ds
 \B_1=
\left[g+|\N u|^{q-1}+|\N v|^{q-1}\right]\ki \qquad \mbox{ and } \qquad 
     \B_2=
\left[
|\N u|^{2}+|\N v|^{2}
\right]^{\frac{N(q-(p-1))+2q-p}{2}}\ki.
\end{array}
\end{equation*}
We conclude the proof applying \cref{lemell1} with $\ro=a$, $\nu=2$ to the function $G_k(u-v)$.
\end{proof}
 
\medskip

Using the same ideas, adapted to the case of $L^1$ data and low values of $q$,  we prove now \cref{teo33} and  \cref{thm26}.
\medskip

\begin{proof}[Proof of \cref{teo33} and \cref{thm26}]
Before getting into the real proof, we recall \eqref{h2} and observe that
\begin{equation}\label{gradreg2}
\big[g+|\N u|^{q-1}+|\N v|^{q-1}\big]\ki\in L^{d}(Q_T).
\end{equation}
Indeed, \cref{lemmaq} provides us that  $|\N \up |,\,|\N  \vm  |\in L^q(Q_T)$ for every $q<2-\frac{N}{N+1}$. So
$|\N  \up |^{q-1}$ and $|\N \vm  |^{q-1}$ belong to $L^r(Q_T)$ for every $r>N+2$. Furthermore, being $(u-v)_+ $ a subsolution itself and 
reasoning as in the just mentioned Lemma with $T_k((\zn)_+ )$, then $(u-v)_+ $ inherits the $L^q(0,T;W^{1,q}(\Omega))$ regularity for every $q<2-\frac{N}{N+1}$. We underline that we use the fact that $\omn$ is assumed to be uniformly bounded in $n$ - it will soon proved - and also converging to $0$. Then, the regularity in \eqref{gradreg2} follows since
\begin{equation*}
\ds |\N v|\ki\ds \le |\N v|\chi_{\{v<0\}}+|\N v|\chi_{\{0<v<u\}} \le |\N  v_-  |+|\N  u_+ |+|\N (u-v)_+|,
\end{equation*}
and
\begin{equation*}
\ds|\N u|\ki \le\ds|\N  v_-  |+|\N  u_+ |+|\N (u-v)_+ |.
\end{equation*}

\smallskip

We use the same notation of  \cref{teorin2} (see \eqref{unvn} and \eqref{zn}); in particular wee take again $S(u)=S_n(u)$ in \eqref{zetaeps2} 
and $S(v)=S_n(v)$ in the formulation of the supersolution, and we set 
\[
\vp=\vp(\zn)=\frac{1}{\mu^\lm}-\frac{1}{( G_k(\zn) +\mu)^\lm} \quad\text{with}\quad\lm=\frac{N+1}{N}(p-q)-1\quad\text{and}\quad \mu >0
\]
 getting
\begin{equation}\label{uv2}
\begin{array}{c}
\ds \integrale\Vp_k(\zn(t))\,dx+
\lm\alpha\iint_{Q_t}\frac{|\N ( G_k(\zn) )|^2}{\left( G_k(\zn) +\mu\right)^{\lm+1}}\left[
|\N {u}_n|^{2}+|\N {v}_n|^{2}
\right]^{\frac{p-2}{2}}\,dx\,ds
\\
[4mm]
\ds
\le \frac{c_1}{\mu^\lm}\iint_{Q_t} \left|\N  G_k(\zn) \right|\B_{n,1}
\,dx\,ds+\frac{\Rn}{\mu^\lm}
\end{array}
\end{equation}
where $\B_{n,1}$  and 
$\Rn$ have been defined in   \eqref{BBn} and \eqref{I}, respectively.

Once again, we have that $\ds \lim_{n\to\infty}\Rn=0$;  the proof of this fact   is quite similar to the one contained in \cref{teorin2}. We just recall  the decompositions \eqref{omn}--\eqref{I} and that, using  the current choice of $\vp(\cdot)$ and the asymptotic energy condition \eqref{ET}, we deduce that  $I_1+I_2=\om_n $. 
The terms $I_3$ and $I_4$ follow as in \cref{teorin2}, just observing that now $f$ only belong to $L^1 (Q_T)$, using again \eqref{ET}.

\medskip

The uniform boundedness of the right hand side  in \eqref{uv2} follows from the above remark by   \eqref{gradreg2}. % and the  \cref{rmkh2}.\\

\medskip

Now, let the parameter $a\le p$ be such that $\frac{a(2-p)}{2-a}=q$, i.e. $a=\frac{2q}{q-p+2}$. Then, 
recalling the inequality in \eqref{uv2}, we obtain
\begin{equation}\label{rhs}
\begin{array}{c}
\ds
\iint_{Q_t}\frac{\left|\N  G_k(\zn) \right|^a}{\left( G_k(\zn) +\mu\right)^{a\frac{\lm+1}{2}}}\,dx\,ds
\le \frac{c}{\mu^{\frac{\lm a}{2}}}\left(
\iint_{E_{n,k}}\left|\N  G_k(\zn) \right|\B_{n,1}
\,dx\,ds+\omn
\right)^{\frac{a}{2}}
\left(\iint_{E_{n,k}}\B_{n,2}\,dx\,ds
\right)^{\frac{2-a}{2}}
\end{array}
\end{equation}
for $\B_{n,1},\,\B_{n,2}$ as in \eqref{BBn}, $E_{n,k}$ as in \eqref{En1}. 
In particular, \eqref{uv2} provides us with
\begin{equation}\label{qwerty}
\begin{array}{c}
\ds \integrale\Vp_k (\zn(t))\,dx
\le \frac{\bar{\ga}}{\mu^\lm}\iint_{E_{n,k}} \left|\N G_k(\zn)\right|\B_{n,1}
\,dx\,ds+\frac{ \omn}{\mu^\lm}.
\end{array}
\end{equation}
Furthermore, the definition of $\Vp_k(\cdot)$ implies
\begin{equation}\label{poiu}
\Vp_k (w)\ge \frac{c}{\mu^\lm} G_k(w) +c,
\end{equation}
with $c$ independent from $\mu$, so that 
\[
\integrale\Vp_k (\zn(t))\,dx\ge \frac{c}{\mu^\lm}\integrale G_k(\zn(t)) \,dx+c.
\]
We use \eqref{poiu} in \eqref{qwerty}, so that 
\begin{equation*}
\integrale  G_k(\zn(t)) \,dx\le  c 
\iint_{E_{n,k}} \left|\N  G_k(\zn) \right|\B_{n,1}\,dx\,ds+\omn+c\mu^\lm
\end{equation*}
that becomes, taking the limits as $\mu\to 0$, and then as $n\to +\infty$,
\begin{equation*}%\label{rhs2}
\integrale  G_k(u(t)-v(t)) \,dx\le 
c\iint_{E_{n,k}} \left|\N  G_k(u-v) \right|\B_1\,dx\,ds, 
\end{equation*}
where  the last convergence follows thanks to  \cref{lemmaq}.

%{ that the last convergence holds the sequences    $\B_{n,1} $ and $\left|\N  G_k(\zn) \right| $ are compact in $ L^{q'}(Q_T)$ and $ L^{q}(Q_T)$, respectively.}\\
Furthermore, letting $n\to\infty$ also in \eqref{rhs}, we get
\begin{equation*} %\label{rhs}
\begin{array}{c}
\ds
\iint_{Q_t}\frac{\left|\N  G_k(u-v) \right|^a}{\left[G_k(u-v) +\mu\right]^{a\frac{\lm+1}{2}}}\,dx\,ds
\le \frac{c}{\mu^{\frac{\lm a}{2}}}\left(
\iint_{Q_t}\left|\N  G_k(u-v) \right|\B_1
\,dx\,ds
\right)^{\frac{a}{2}}
\left(\iint_{E_k}\B_2\,dx\,ds
\right)^{\frac{2-a}{2}}
\end{array}
\end{equation*}
with $E_k$ as in \eqref{EE} and 
\begin{equation*} %\label{B1B2}
 \B_1=\left[g+
|\N u|^{q-1}+|\N v|^{q-1}\right]\ki
,\qq
 \B_2=
\left[
|\N u|^{2}+|\N v|^{2}
\right]^\frac{q}{2}\ki.
\end{equation*}
We now apply \cref{leml1} with 
\[
\ro=a,\quad \nu=\frac{\lm a}{2},\quad m=1,\quad\ga=a\frac{1+\lm}{2}
\]
%\footnote{qui c'è un $\ga$ che non è il $\ga$ dello spazio di energia. fa nulla, tanto il lemma si usa solo con dati L1?}
and
\[
\begin{array}{c}
\ds
M=c\left(
\iint_{Q_t}\left|\N  G_k(u-v) \right|\B_1
\,dx\,ds
\right)^{\frac{a}{2}}
\left(\iint_{E_k}\B_2\,dx\,ds
\right)^{\frac{2-a}{2}},
\end{array}
\qquad 
L=c\iint_{Q_t}\left|\N  G_k(u-v) \right|\B_1
\,dx\,ds.
\]
In particular, the estimate in \eqref{gradb} holds with
\[
b=a\frac{N(1-\lm)+2}{2(N+1)}=q
\]
by the definitions of $a,\,\lm$ and
\[
\begin{array}{c}
\ds
\|\N  G_k(u-v) \|_{L^q(Q_T)}
\le c\left(\iint_{Q_t} \left|\N  G_k(u-v) \right|\B_1\,dx\,ds\right)\left(\iint_{E_k}\B_2\,dx\,ds
\right)^{\frac{(2-a)(N+1)}{a(N+2)}}.
\end{array}
\]
We apply H\"older's inequality with $(q,q')$ obtaining
\[
\begin{array}{c}
\ds
\|\N  G_k(u-v) \|_{L^q(Q_T)}
\le c \|\N  G_k(u-v) \|_{L^q(Q_T)}
\left(\iint_{Q_t} \B_1^{\frac{q}{q-1}}\,dx\,ds\right)^{\frac{q-1}{q}}\left(\iint_{E_k}\B_2\,dx\,ds
\right)^{\frac{(2-a)(N+1)}{a(N+2)}}.
\end{array}
\]
Observe that, since  $\frac{q}{q-1}\searrow \frac{p(N+1)-N}{p(N+1)-(2N+1)}$ as $q\nearrow p-\frac{N}{N+1}$, the integral involving $\B_1$ is bounded (up to choose $q$ closer to the threshold). 
Then we conclude by applying \cref{lemell1} with $\ro=q$.
\end{proof}

\section{Proofs in the case $2\le p<N$}\label{pge2}

\setcounter{equation}{0}
\renewcommand{\theequation}{\thesection.\arabic{equation}}
\numberwithin{equation}{section}

We start by proving the results for $p=2$, since their proofs are different from those of the case $p>2$. 
\subsection{The case $p=2$}

\begin{proof}[Proof of \cref{comp2}]
We follow the same notation that we have used for the proof of \cref{teorin2}, by defining  $\un,\,\vn$ as in \eqref{unvn}. Thus we   consider the inequalities in \eqref{renf} satisfied by the sub/supersolutions with $S(u)=u_n $ and $S(v)=v_n$ respectively so that,   we have
\begin{equation*}
\begin{array}{c}
\ds
\int_\Omega u_n(t)\varphi(t)\,dx+\iint_{Q_t}  A(s,x)\N \un\cdot \N\vp+A(s,x)\N u\cdot \N u \,\te'_n(u)\vp\,dx\,ds\\
[4mm]\ds
\le \iint_{Q_t} H(t,x,\N u)\tenu\varphi\,dx\,ds+\int_\Omega u_n(0)\varphi(0)\,dx
\end{array}
\end{equation*}
and
\begin{equation*}
\begin{array}{c}
\ds
\int_\Omega (1-\eps)v_n(t)\varphi(t)\,dx+\iint_{Q_t}  A(s,x)\N ((1-\eps)\vn)\cdot \N\vp+A(s,x)\N ((1-\eps)v)\cdot \N v \,\te'_n(v)\vp\,dx\,ds\\
[4mm]\ds
\ge (1-\eps) \iint_{Q_t} H(t,x,\N v)\tenv\varphi\,dx\,ds+(1-\eps)\int_\Omega v_n(0)\varphi(0)\,dx
\end{array}
\end{equation*}
where the inequality related to the supersolution has been multiplied by $(1-\eps)$, for $\eps\in (0,1)$. 
Then, taking into account the difference between the inequalities above, we get 
\begin{equation*}
\begin{array}{c}
\ds
\int_\Omega\left[{u}_n(t)-(1-\eps){v}_n(t)\right]\varphi(t)\,dx+\iint_{Q_t}  A(s,x)\N \left({u}_n-(1-\eps){v}_n\right)\cdot \N\vp\,dx\,ds\\
[4mm]\ds+
\iint_{Q_t} 
\big[A(s,x)\N u\cdot \N u \,\te'_n(u)-A(s,x)\N ((1-\eps)v)\cdot \N v \,\te'_n(v)\big]\vp\,dx\,ds\\
[4mm]\ds
\le  \iint_{Q_t}\big[ H(t,x,\N u)\tenu-(1-\eps)H(t,x,\N v)\tenv\big]\varphi\,dx\,ds+\int_\Omega\zne(0)\varphi(0)\,dx.
\end{array}
\end{equation*}
%Arguing as in Theorem \ref{teorinL12}, 
We use the hypothesis \eqref{dec} in order to we rewrite the above inequality as 
\begin{equation}\label{rn2}
\begin{array}{c}
\ds
\int_\Omega\left({u}_n(t)-(1-\eps){v}_n(t)\right)\varphi(t)\,dx+\iint_{Q_t}  A(s,x)\N \left({u}_n-(1-\eps){v}_n\right)\cdot \N\vp\,dx\,ds\\
[4mm]\ds
\le  \iint_{Q_t}\left[ H_1(s,x,\N \un)-(1-\eps)H_1(s,x,\N \vn)\right]\varphi+\left[ H_2(s,x,\N \un)-(1-\eps)H_2(s,x,\N \vn)\right]\varphi\,dx\,ds+\Rn\,,
\end{array}
\end{equation}
where 
\[
\begin{split}
\Rn&=
\iint_{Q_t} 
\big[A(s,x)\N u\cdot \N u \,\te'_n(u)-A(s,x)\N ((1-\eps)v)\cdot \N v \,\te'_n(v)\big]\vp\,dx\,ds\\
&\quad+\iint_{Q_t}\big[ H(s,x,\N u)\tenu-(1-\eps)H(s,x,\N v)\tenv-H(s,x,\N \un)+(1-\eps)H(s,x,\N \vn)\big]\varphi\,dx\,ds\\
&\quad+
\int_\Omega\zne(0)\varphi(0)\,dx.
\end{split}
\]

Let us observe  that the conditions \eqref{H2.1} and \eqref{H2.2} imply that
$$\begin{array}{c}
\ds
H_2(s,x,\N \un)-(1-\varepsilon)H_2(s,x,\N \vn)\\
[4mm]\ds
\le |H_2(s,x,\N \un)-H_2(s,x,(1-\varepsilon)\N \vn)|
 +H_2(s,x,(1-\varepsilon)\N \vn)-(1-\varepsilon)H_2(s,x,\N \vn)\\
 [4mm]\ds 
\le c_2|\N \left({u}_n-(1-\eps){v}_n\right)|
\le  c  \varepsilon \bigg[ \frac{|\N  \left({u}_n-(1-\eps){v}_n\right)|^q}{\varepsilon^q} +1\bigg]\,,
\end{array}
$$
byYoung's inequality.
On the other hand, since $H_1(t,x,\xi)$ satisfies \eqref{H1.1} (i.e. the convexity  assumption with respect to  $\xi$), we have 
\[
H_1(s,x,\N \un)-(1-\varepsilon)H_1(s,x,\N \vn)\le \varepsilon H_1\biggl(t,x,\frac{\N \left({u}_n-(1-\eps){v}_n\right)}{\varepsilon}\biggl).
\]
Finally, the growth assumption on $H_1(s,x,\xi)$ contained in \eqref{H111} allows us to improve \eqref{rn2} as
\begin{equation*}
\begin{array}{c}
\ds
\int_\Omega\left({u}_n(t)-(1-\eps){v}_n(t)\right)\varphi(t)\,dx+\iint_{Q_t}  A(s,x)\N \left({u}_n-(1-\eps){v}_n\right)\cdot \N\vp\,dx\,ds\\
[4mm]\ds
\le  \eps\iint_{Q_t}\left(c_1\frac{|\N \left({u}_n-(1-\eps){v}_n\right)|^q}{\varepsilon^q} + \tilde f\right) \,dx\,ds+\Rn
\end{array}
\end{equation*}
where $\tilde f=f+c$. 
In particular, the inequality above can be written in terms of the function
\begin{equation}\label{zneh}
\zne (t,x)=\frac{e^{- t}}{\eps}\big({u}_n (t,x) -(1-\eps){v}_n(t,x) \big)
\end{equation}
as
\begin{equation}\label{finz}
\begin{array}{c}
\ds
\int_\Omega\zne(t)\varphi(t)\,dx+\iint_{Q_t}  \zne\vp\,dx\,ds+\iint_{Q_t}  A(s,x)\N \zne\cdot \N\vp\,dx\,ds\\
[4mm]\ds
\le   \iint_{Q_t} \bigg(c_1   {|\N z^\varepsilon_n|^q}  +  \tilde f\bigg) \vp \,dx\,ds+\frac{\Rn}{\eps}.
\end{array}
\end{equation}
We consider the inequality in \eqref{finz} with 
$$\ds \vp(\zne)=\int_0^{ G_k(\zne) }(\mu+|w|)^{\si-2}\,dw,\quad \mu>0,$$ 
(again, we recall the density results in \cite[Proposition $4.2$]{PPP}) getting
\begin{equation*}
\begin{array}{c}
\ds
\int_{\Omega}  \Vp_k(\zne(t))\,dx
+\iint_{Q_t}  \zne\vp\,dx\,ds
+{\al}\iint_{Q_t}|\N \Psi_k( G_k(\zne) )|^2\,dx\,ds
\\
[4mm]\ds
\le c_1  \iint_{Q_t} |\N  G_k(\zne) |^q\biggl(\int_0^{ G_k(\zne) }(\mu+|w|)^{\sigma-2}\,dz\biggr)\,dx\,ds
\\
[4mm]\ds
+\iint_{Q_t} |\tilde{f}|\chi_{\{| \tilde{f}|>k\}} \biggl(\int_0^{ G_k(\zne) }(\mu+|w|)^{\sigma-2}\,dz\biggr) \,dx\,ds+\iint_{Q_t} |\tilde{f}|\chi_{\{| \tilde{f}|\le k\}} \vp \,dx\,ds
+\frac{\Rn}{\eps}\,,
\end{array}
\end{equation*} 
with
\[
\Psi_{\mu}( G_k(\zne) )=\int_0^{ G_k(\zne) } (\mu+|w|)^{\frac{\sigma-2}{2}}\,dw \qquad\text{and}\qquad\Vp_k(\zne(t)))=\int_0^{ G_k(\zne(t)) }\vp(w)\,dw.
\]
Observe that since $\vp$ is supported where $\zne \geq k$, then  
\[
\ds \iint_{Q_t}  \zne\vp\,dx\,ds-\iint_{Q_t} |\tilde{f}|\chi_{\{| \tilde{f}|\le k\}} \vp\,dx\,ds\ge
k\iint_{Q_t}  \vp\,dx\,ds-k\iint_{Q_t}\vp \,dx\,ds\ge 0\,,
\]
and we are reduced to study
\begin{equation*}
\begin{array}{c}
\ds
\int_{\Omega}  \Vp_k(\zne(t))\,dx
+{\al}\iint_{Q_t}|\N \Psi_k( G_k(\zne) )|^2\,dx\,ds
\\
[4mm]\ds
\le c_1  \iint_{Q_t} |\N  G_k(\zne) |^q\biggl(\int_0^{ G_k(\zne) }(\mu+|w|)^{\sigma-2}\,dw\biggr)\,dx\,ds
+\iint_{Q_t} |\tilde{f}|\chi_{\{| \tilde{f}|>k\}} \biggl(\int_0^{ G_k(\zne) }(\mu+|w|)^{\sigma-2}\,dw\biggr) \,dx\,ds 
+\frac{\Rn}{\eps}.
\end{array}
\end{equation*} 
We just note that, by definitions of $\te_n(\cdot)$, $\vp(\cdot)$ and thanks to \eqref{aa1}, \eqref{pot}, the proof of $\dys \lim_{n\to\infty}\Rn=0$ follows reasoning as in \cref{teorin2}.\\

We observe that the definition of $\Psi_k(\cdot)$ combined with H\"older's inequality with indices $\left(\frac{1}{2-q},\frac{1}{q-1}\right)$ and also an application Young inequality with $\left(\frac{2}{q},\frac{2}{2-q}\right)$ leads   to 
\[
\begin{array}{c}
\ds
 \iint_{Q_t} |\N  G_k(\zne) |^q\biggl(\int_0^{ G_k(\zne) }(\mu+|w|)^{\sigma-2}\,dw \biggr)\,dx\,ds
\\[4mm]\ds
\le c  \iint_{Q_t}|\N \Psi_k( G_k(\zne) )|^q\biggl(\int_0^{ G_k(\zne) }(\mu+|w|)^{(\sigma-2)\frac{2-q}{2}}\,dw \biggr)\,dx\,ds\\
[4mm]
\ds\le c  \iint_{Q_t} |\N \Psi_k( G_k(\zne) )|^q 
 |\Psi_k( G_k(\zne) )|^{2-q} G_k(\zne) ^{q-1}\,dx\,ds\\
 [4mm]
 \ds\le  \frac{{\al}}{2}\iint_{Q_t} |\N \Psi_k( G_k(\zne) )|^2\,dx\ ds+c \iint_{Q_t}
  |\Psi_k( G_k(\zne) )|^{2} G_k(\zne) ^{\frac{2(q-1)}{2-q}}\,dx\,ds.
\end{array}
\]
We now focus on the term involving the source $f$. 
We recall the estimate in \cref{teorin2} and apply H\"older's inequality with $(m,m')$ and $(r,r')$ getting
\begin{equation*}
\begin{array}{c}
\ds
\iint_{Q_t} |\tilde{f}|\chi_{\{| \tilde{f}|>k\}} \biggl(\int_0^{ G_k(\zne) }(\mu+|w|)^{\sigma-2}\,dw\biggr) \,dx\,ds \\
[4mm]\ds \le c\iint_{Q_t}|\tilde{f}|\chi_{\{| \tilde{f}|>k\}} \biggl(\int_0^{ G_k(\zne) }(\mu+|w|)^{\frac{\sigma-2}{2}}\,dw \biggr)^{\frac{2}{\sigma'}} \,dx\,ds\\
[4mm]\ds
\le c\| |\tilde{f}|\chi_{\{| \tilde{f}|>k\}}\|_{L^r(0,T;L^m(\Omega))}\left\|\Psi_k( G_k(\zne) )\right\|_{L^{2r'\frac{\sigma-1}{\si}}(0,t;L^{2m'\frac{\sigma-1}{\si}}(\Omega))}^{2\frac{\sigma-1}{\si}}
\end{array}
\end{equation*}
where $r,\,m$ verifies \eqref{F1}. Finally, we gather together  the estimates above and find that
\begin{equation}\label{k}
\begin{array}{c}
\ds
\int_{\Omega}  \Vp_k(\zne(t))\,dx+\frac{{\al}}{2}\iint_{Q_t}|\N \Psi_k( G_k(\zne) )|^2\,dx\,ds
\le c \iint_{Q_t}
  |\Psi_k( G_k(\zne) )|^{2} G_k(\zne) ^{\frac{2(q-1)}{2-q}}\,dx\,ds
  \\
[4mm]\ds
+c\| |\tilde{f}|\chi_{\{| \tilde{f}|>k\}}\|_{L^r(0,T;L^m(\Omega))} \|\Psi_k( G_k(\zne) )\|_{L^{2r'\frac{\sigma-1}{\si}}(0,t;L^{2m'\frac{\sigma-1}{\si}}(\Omega))}^{2\frac{\sigma-1}{\si}}+\omn.
\end{array}
\end{equation} 
Then,  since $\si+2\frac{q-1}{2-q}=2\frac{N+\si}{N}$ (indeed, $\Psi_k(\zne)\le\frac{c}{\eps}(|u|^{\frac{\si}{2}}+|u|^{\frac{\si}{2}})$) and thanks to \eqref{pot}, the energy integral in the left hand side  is uniformly bounded.

Moreover by \cref{teoGN} we have that 
\[
\iint_{Q_t}
  |\Psi_k( G_k(\zne) )|^{2} G_k(\zne) ^{\frac{2(q-1)}{2-q}}\,dx\,ds\le c\| G_k(\zne) \|_{L^{\infty}(0,t;L^{\si}(\Omega))}^{q-1}\int_0^t\|\N \Psi_k( G_k(\zne(s)) )\|_{L^{2}(\Omega)}^{2}\,ds
\]
and, reasoning as in \eqref{frn}, we get
\[
 \|\Psi_k( G_k(\zne) )\|_{L^{2r'\frac{\sigma-1}{\si}}(0,t;L^{2m'\frac{\sigma-1}{\si}}(\Omega))}^{2\frac{\sigma-1}{\si}}\le c\| G_k(\zne) \|_{L^{\infty}(0,t;L^{\si}(\Omega))}^{\frac{\si}{2}(y-2)}\int_0^t\|\N \Psi_k( G_k(\zne(s)) )\|_{L^{2}(\Omega)}^{2}\,ds,
\]
so that the inequality \eqref{k} reads as 
\begin{equation*}
\begin{array}{c}
\ds
\int_{\Omega} \Vp_k(\zne(t))\,dx
+\frac{\al}{2}\iint_{Q_t}|\N \Psi_{\mu}( G_k(\zne) )|^2 \,dx\,ds
\\
[4mm]\ds
\le C_1 \left[\| G_k(\zne) \|_{L^{\infty}(0,t;L^{\si}(\Omega))}^{q-1}+\| G_k(\zne) \|_{L^{\infty}(0,t;L^{\si}(\Omega))}^{\frac{\si}{2}(y-2)}\right]\int_0^t\|\N \Psi_k( G_k(\zne(s)) )\|_{L^{2}(\Omega)}^{2}\,ds\\
[4mm]\ds
+C_2\| |\tilde{f}|\chi_{\{| \tilde{f}|>k\}}\|_{L^r(0,T;L^m(\Omega))}^{\frac{y\si}{y\si-2(\si-1)}}+\int_{\Omega} \Vp_k(v_0)\,dx
+\omn
\end{array}
\end{equation*}
for any $\mu>0$, where we have used that  $\zne(0)\le v_0$. \\
 Now, reasoning as in the proof of the a priori estimates contained in  \cite{Ma}. We fix a value $\de_0$ such that \\$\max\left\{
\de_0^{\frac{q-1}{\si}},\de_0^{\frac{y-2}{2}}
\right\}=  \frac1{2C_1 }  \frac{\al}{2}$ and we take $k\ge k_0$ such that
\[
\integrale  G_k(v_0) ^{\sigma}\,dx+C_2\| |\tilde{f}|\chi_{\{| \tilde{f}|>k\}}\|_{L^r(0,T;L^m(\Omega))}^{\frac{y\si}{y\si-2(\si-1)}}<\de_0
\]
for any $k\ge k_0$. We also define
\[
T^*=\sup\left\{\tau>0: \quad \|G_k(z_n^\eps(s))\|_{L^\si(\Omega)}^\si\le \de_0\quad \forall s\le \tau  \right\}
\]
which is strictly positive by \eqref{csi} and since $u_0\leq v_0$ in $\Omega$.  Note also that $T^*$ continuously depends on $n$ by \eqref{csi}.  \\
Then, for $k\ge k_0$ and for any $t\le T^*$ we have
\begin{equation*}
\sup_{t\in [0,T^*]}\int_{\Omega} \Vp_k(z_n^\eps(t))\,dx
<\de_0
\end{equation*}
and, letting $\mu\to 0$, we deduce
\begin{equation}\label{T}
\sup_{t\in [0,T^*]}\int_{\Omega} G_k(z_n^\eps(t))^\si\,dx
<\de_0.
\end{equation}
Now, if $T^*<T$, then \eqref{T} would be in contrast with both the continuity regularity in \eqref{csi} and the definition of $T^*$, so \eqref{T} holds up to $T$.\\
We thus deduce a bound, uniform in $\varepsilon$, for the function ${z_n}^\varepsilon(t)$ in $L^{\sigma}(\Omega)$. Indeed, we have
\begin{equation*}%\label{bd}
\sup_{t\in [0,T]}\integrale  z_n^\varepsilon(t) ^{\sigma}\,dx\le \de_0+ k_0^{\sigma}T|\Omega|
\end{equation*}
and then, letting $n\to\infty$ and recalling the definition of $\zn$ in \eqref{zneh}, leads to 
\[
\integrale [u(t)-(1-\varepsilon)v(t) ]^{\sigma}\,dx\le \varepsilon^{\sigma} c
\]
which, letting $\varepsilon\to 0$, implies $u\le v$ in $Q_T$ and thus the assertion is proved.
%where 
%\begin{equation}\label{zeh}
% z^\varepsilon (s,x) =\frac{e^{-s}}{\varepsilon}\big(u(s,x)-(1-\varepsilon)v(s,x)\big).
%\end{equation}
\end{proof}

\medskip

Next we prove \cref{teorinL12}. 

\medskip

\begin{proof}[Proof of \cref{teorinL12}]
We start recalling the inequality in \eqref{finz}, with $\zne$ defined in \eqref{zneh}, and we set
$$ \vp(\zne)=1-\frac{1}{(1+ G_k(\zne) )^\mu},\qquad \mbox{with } \quad \mu=\frac{N+1}{N}(2-q)-1<1,
$$
so we get
\begin{equation*}  
\begin{array}{c}
\ds
\integrale \Vp(\zne(t))\,dx
+\mu\al\iint_{Q_t}\frac{|\N  G_k({z}^\eps_n) |^2}{(1+ G_k({z}^\eps_n) )^{\mu+1}}\,dx\,ds+\iint_{Q_t}\zne \left[1-\frac{1}{(1+ G_k({z}^\eps_n) )^\mu}\right]\,dx\,ds
\\
[4mm]
\ds
\le c_1e^{(q-1) T}\iint_{Q_t}|\N  G_k({z}^\eps_n) |^q\left[1-\frac{1}{(1+ G_k({z}^\eps_n) )^\mu}\right]\,dx\,ds+\iint_{Q_t}| \tilde{f}|\chi_{\{|f|>k\}}\left[1-\frac{1}{(1+ G_k({z}^\eps_n) )^\mu}\right]\,dx\,ds \\
[4mm]\ds
+\iint_{Q_t}| \tilde{f}|\chi_{\{|f|\le k\}}\left[1-\frac{1}{(1+ G_k({z}^\eps_n) )^\mu}\right]\,dx\,ds+\integrale \Vp(\zne(0))\,dx+\Rn,
\end{array}
\end{equation*}
where $\Phi (w)= \int_0^w \vp (y) dy$. 
Again, we observe that  
\[
\iint_{Q_t}\zne \left[1-\frac{1}{(1+ G_k({z}^\eps_n) )^\mu}\right]\,dx\,ds-\iint_{Q_t}| \tilde{f}|\chi_{\{|f|\le k\}}\left[1-\frac{1}{(1+ G_k({z}^\eps_n) )^\mu}\right]\,dx\,ds\ge 0
\]
  so we drop it,  and we just deal with
\begin{equation} \label{in9}
\begin{array}{c}
\ds
\integrale \Vp(\zne(t))\,dx
+\mu\al\iint_{Q_t}\frac{|\N  G_k({z}^\eps_n) |^2}{(1+ G_k({z}^\eps_n) )^{\mu+1}}\,dx\,ds
\\
[4mm]
\ds
\le c_1e^{(q-1) T}\iint_{Q_t}|\N  G_k({z}^\eps_n) |^q\left[1-\frac{1}{(1+ G_k({z}^\eps_n) )^\mu}\right]\,dx\,ds+\iint_{Q_t}| \tilde{f}|\chi_{\{|f|>k\}}\left[1-\frac{1}{(1+ G_k({z}^\eps_n) )^\mu}\right]\,dx\,ds \\
[4mm]\ds
+\integrale \Vp(\zne(0))\,dx+\Rn.
\end{array}
\end{equation}
Our purpose is to recover  an a priori estimate for $\zne$. We begin applying Young's inequality with $\left(\frac{2}{q},\frac{2}{2-q}\right)$ to the first integral in the right hand side  of \eqref{in9} obtaining
\begin{equation*}
\begin{array}{c}
\ds
c_1e^{(q-1) T}\iint_{Q_t}|\N  G_k(\zne) |^q\left[1-\frac{1}{(1+ G_k(\zne) )^\mu}\right]\,dx\,ds\\
[4mm]
\ds
\le \frac{\mu \al}{2}\iint_{Q_t}\frac{|\N  G_k(\zne) |^2}{(1+ G_k(\zne) )^{\mu+1}}\,dx\,ds
+ 
c \iint_{Q_t}(1+ G_k(\zne) )^{\frac{q(\mu+1)}{2-q}}\left[1-\frac{1}{(1+ G_k(\zne) )^\mu}\right]^{\frac{2}{2-q}}\,dx\,ds.
\end{array}
\end{equation*}
We observe that, since $\mu<1$, it holds that 
$1-\frac{1}{(1+ G_k(\zne) )^\mu}\leq \frac{ G_k(\zne) }{1+ G_k(\zne) }$,  and since  $\frac{2}{2-q}>1$, it follows
$
\left(\frac{ G_k(\zne) }{1+ G_k(\zne) }\right)^{\frac{2}{2-q}}\le G_k(\zne)
$
 and thus we deduce the uniform boundedness 
\begin{equation*}  
\begin{array}{c}
\ds
\integrale \Vp(\zne(t))\,dx
+\mu\al\iint_{Q_t}\frac{|\N  G_k({z}^\eps_n) |^2}{(1+ G_k({z}^\eps_n) )^{\mu+1}}\,dx\,ds
\\
[4mm]
\ds
\le c e^{(q-1) T}\iint_{Q_t} |G_k({z}^\eps_n)|\,dx\,ds+\iint_{Q_t}| \tilde{f}|\chi_{\{|f|>k\}}\,dx\,ds  
+\integrale \Vp(\zne(0))\,dx+\Rn,
\end{array}
\end{equation*}
since, as already observed,  the asymptotic condition \eqref{ET} takes the place of \eqref{pot} in the proof that $\ds \lim_{n\to\infty}\Rn=0$. \\
In particular, this means that the energy term above is uniformly bounded in $n$.

Furthermore, since $\ds q\frac{\mu+1}{2-q}=q\frac{N+1}{N}$ is the Gagliardo-Nirenberg exponent associated  to the spaces
$
 L^2(0,T;H_0^1(\Omega))\cap L^\infty(0,T; L^{\frac{2}{1-\mu}}(\Omega))
$, 
(see \cref{teoGN}) we have that
\begin{equation*}
\begin{array}{c}
\ds
\iint_{Q_T}(1+ G_k(\zne) )^{\frac{q(\mu+1)}{2-q}}\left[1-\frac{1}{(1+ G_k(\zne) )^\mu}\right]^{\frac{2}{2-q}}\,dx\,dt
\le 
\iint_{Q_T}(1+ G_k(\zne) )^{\frac{q(\mu+1)}{2-q}-1} G_k(\zne) \,dx\,dt\\
[4mm]\ds
\le  {c_{GN}} \| G_k(\zne) \|_{L^\infty(0,T;L^{1}(\Omega))}^{\frac{p}{N}}\iint_{Q_T}\frac{|\N  G_k(\zne) |^2}{(1+ G_k(\zne) )^{\mu+1}}\,dx\,dt\,.
\end{array}
\end{equation*}
%where the first inequality is due to \eqref{disG}.\\
Then \eqref{in9}   becomes
\[
\begin{array}{c}
\ds
\integrale \Vp(\zne(t))\,dx
+\frac{\mu\al}{2}\iint_{Q_t}\frac{|\N   G_k(\zne) |^2}{(1+ G_k(\zne) )^{\mu+1}}\,dx\,ds
\\
[4mm]
\ds
\le \bar{c}\| G_k(\zne) \|_{L^\infty(0,T;L^{1}(\Omega))}^{\frac{2-q}{N}}\iint_{Q_T}\frac{|\N  G_k(\zne) |^2}{(1+ G_k(\zne) )^{\mu+1}}\,dx\,dt
+\iint_{Q_t}| \tilde{f}|\chi_{\{|f|>k\}}\,dx\,ds+\integrale  G_k(v_0) \,dx
+\omn.
\end{array}
\]
%This means that, letting $n\to \infty$, we are left with the study of the following inequality:
%\begin{equation}\label{ptpt}
%\begin{array}{c}
%\ds
%\integrale \Vp(z^\eps(t))\,dxs
%+\frac{\mu\al}{2}\iint_{Q_t}\frac{|\N  G_k(z^\eps)  |^2}{(1+ G_k(z^\eps) )^{\mu+1}}\,dx\,ds\\
%[4mm]
%\ds
%\le  \bar{c}\| G_k(z^\eps) \|_{L^\infty(0,t;L^{1}(\Omega))}^{\frac{2-q}{N}}\iint_{Q_t}\frac{|\N  G_k(z^\eps) |^2}{(1+ G_k(z^\eps) )^{\mu+1}}\,dx\,ds+\integrale  G_k(v_0) \,dx+\iint_{Q_t}| \tilde{f}|\chi_{\{|\tilde{f}|>k\}}\,dx\,ds
%.
%\end{array}
%\end{equation}

We observe that the function $\Vp(z)$ can be estimated from below as $\Vp( w)>C_1\min\{ G_k (w) ,G_k (w)^2\}$, for a  certain $C_1>0$,   from which
\begin{equation}\label{bohh}
\begin{split}
\integrale  G_k(z_n^\eps(t))\,dx & \le \int_{\{  G_k(z_n^\eps(t))>1\}} G_k(z_n^\eps(t))\,dx+|\Omega|^{\frac{1}{2}}\left( \int_{\{ G_k(z_n^\eps(t))\le 1\}}  G_k(z_n^\eps(t))^2\,dx\right)^\frac{1}{2}\\
&\le \frac{1}{C_1}\integrale \Vp(z_n^\eps(t))\,dx+\left(\frac{|\Omega|}{C_1}\right)^\frac{1}{2}\left(\integrale \Vp(z_n^\eps(t))\,dx\right)^\frac{1}{2}. 
\end{split}
\end{equation}
We fix a small value $\de_0$ so that the equality
$\ds \frac{\mu\al}{2}=\bar{c}(C_0\de_0^\frac{1}{2})^{\frac{p-q}{N}}$ holds for $\ds C_0=2\max\left\{\frac{1}{C_1},\left(\frac{|\Omega|}{C_1}\right)^\frac{1}{2} \right\}$.
Moreover, for $\de<\de_0$ , we let $k_0$ large enough so that
\begin{equation}\label{***}
\integrale  G_k(v_0) \,dx+
\iint_{Q_t}| \tilde{f}|\chi_{\{|\tilde{f}|>k\}}\,dx\,ds<\de\qquad\forall  k\ge k_0
\end{equation}
and we define
\[
T^*:=\sup\{\tau> 0: \,\|G_k(z_n^\eps(s))\|_{L^{1}(\Omega)}\le C_0\de^\frac{1}{2}, \,\,\forall\, s\le \tau  \} \qquad\forall k\ge k_0.
\]
Notice that $T^*>0$ thanks to the regularity $u,\,v\in C([0,T];L^1(\Omega))$. We underline again that such a continuity regularity implies that $T^*$ continuously depends on $n$.\\
The above choice of $\de_0$ and \eqref{***} imply
\begin{equation}
\integrale \Vp(z^\eps(t))\,dx\le \integrale  G_k(v_0) \,dx+\iint_{Q_t}| \tilde{f}|\chi_{\{|\tilde{f}|>k\}}\,dx\,ds<\de\qquad\forall k\ge k_0,\,\,\forall t\le T^*.
\end{equation}
Therefore, by definition of $C_1$ and $C_0$ and thanks to \eqref{bohh}, we obtain
\begin{equation}\label{contr2}
\begin{split}
\integrale G_k(z_n^\eps(t))\,dx< C_0\de^\frac{1}{2}\q\forall t\le T^*.
\end{split}
\end{equation}
By the continuity regularity $u,\,v\in C([0,T];L^1(\Omega))$ and \eqref{contr2} we deduce that $T^*= T$, since if $T^*< T$ then \eqref{contr2} would be in contrast with the definition of $T^*$ and since $u,\,v\in C([0,T];L^1(\Omega))$.\\
Once we have \eqref{contr2} for $T= T^*$ then we have
\[
\integrale z_n^\eps(t)\,dx\le \integrale G_k(z_n^\eps(t))\,dx+\integrale T_k(z_n^\eps(t))\,dx\le C_0\de^{\frac{1}{2}}+k
\]
which, letting $n\to \infty$ and recalling the definition of $\zn$ in \eqref{zneh}, leads to
\[
\integrale \big(u(t)-(1-\varepsilon)v(t)\big) \,dx\le \eps e^t(C_0\de^{\frac{1}{2}}+k)
\]
and the proof follows once we let $\eps\to 0$. 
\end{proof}

\subsection{The case $2<p<N$}

%\subsection{The case with sharp initial datum}\label{sharp}

Here we prove our results via the \lq\lq convexity\rq\rq\, method.

\proof[Proof of \cref{teoen1}] 
We want to follow the first part of \cref{comp2}. In order to do it, we recall the definitions of $\un, \,\vn$ in \eqref{unvn} and consider the renormalized formulations in \eqref{renf}. We focus on the one related to the supersolution $v$: we consider $S(v)=\vn$ and multiply its inequality by $ (1-\eps)^{p-1}$, we get 
\begin{equation*}
\begin{array}{c}
\ds
\int_\Omega (1-\eps)^{p-1}v_n(t)\varphi(t)\,dx+\iint_{Q_t}  A(x)|\N (1-\eps)\vn|^{p-2}\N((1-\eps)\vn)\cdot \N\vp\,dx\,ds\\
[4mm]\ds
\ge (1-\eps)^{p-1} \iint_{Q_t} H(s,x,\N \vn)\varphi\,dx\,ds
\\
[4mm]\ds-\iint_{Q_t} A(x)
\big[
\tenv |\N ((1-\eps)v)|^{p-2}\N ((1-\eps)v)-|\N ((1-\eps)\vn)|^{p-2}\N ((1-\eps)\vn)
\big]\cdot\N\vp
\,dx\,ds\\
[4mm]\ds
-(1-\eps)^{p-1}\iint_{Q_t} A(x)\N v\cdot \N v \,\te'_n(v)\vp\,dx\,ds
+ (1-\eps)^{p-1} \iint_{Q_t} \big[H(s,x,\N v)\tenv-H(s,x,\N \vn)\big]\varphi\,dx\,ds\\
[4mm]\ds
+(1-\eps)^{p-1}\int_\Omega v_n(0)\varphi(0)\,dx,
\end{array}
\end{equation*}
Then, rescaling in time this last inequality and defining $\vne$ as
$$ \vne(t,x)=(1-\eps)\vn((1-\eps)^{p-2}t,x),$$ 
we obtain 
\begin{equation}\label{hatv}
\begin{array}{c}
\ds
\int_\Omega \vne(t)\varphi(t)\,dx+\iint_{Q_t^\eps}  A(x)|\N \vne|^{p-2}\N\vne\cdot \N\vp\,dx\,ds
\ge (1-\eps)^{p-1} \iint_{Q_t^\eps} H\left((1-\eps)^{p-2}s,x,\frac{\N \vne}{1-\eps}\right)\varphi\,dx\,ds
\\
[4mm]\ds-\iint_{Q_t^\eps} 
\left[A(x)
\left(
\tenv |\N ((1-\eps)v)|^{p-2}\N ((1-\eps)v)-|\N \vne|^{p-2}\N \vne
\right)
\right]\cdot\N\vp
\,dx\,ds\\
[4mm]\ds
-(1-\eps)^{p-1}\iint_{Q_t^\eps} A(x)\N v\cdot \N v \,\te'_n(v)\vp\,dx\,ds
+ (1-\eps)^{p-1} \iint_{Q_t^\eps} \left[H(s,x,\N v)\tenv-H(s,x,\N \vn)\right]\varphi\,dx\,ds\\
[4mm]\ds
+\int_\Omega \vne(0)\varphi(0)\,dx,
\end{array}
\end{equation}
where $Q_t^\eps=\left(0,\frac{t}{(1-\eps)^{p-2}}\right)\times \Omega$ for $0\le t \le T$. 
We want to take into account the difference between
\begin{equation*}%\label{hatu}
\begin{array}{c}
\ds
\int_\Omega u_n(t)\varphi(t)\,dx+\iint_{Q_t}  A(x)|\N \un|^{p-2}\N \un\cdot \N\vp\,dx\,ds\le \iint_{Q_t}H(s,x,\N \un)\varphi\,dx\,ds\\
[4mm]\ds
-\iint_{Q_t}A(x)\N u\cdot \N u \,\te'_n(u)\vp\,dx\,ds
-\iint_{Q_t}A(x)\left(
\tenu |\N u|^{p-2}\N u-|\N \un|^{p-2}\N \un
\right)\\
[4mm]\ds
+ \iint_{Q_t} \left[H(s,x,\N u)\tenu-H(s,x,\N \un)\right]\varphi\,dx\,ds+\int_\Omega u_n(0)\varphi(0)\,dx
\end{array}
\end{equation*}
and \eqref{hatv}: to this aim, we restrict the integrals to the time interval $0\le t\le T$ since \eqref{hatv} holds in $Q_t^\eps\supset Q_t$. 

As in the previous cases, our aim is to write an inequality solved by the following function 
\begin{equation}\label{zneh2}
\zne (t,x)=\frac{e^{-\lm t}}{\eps}\big(\un(t,x)-\vne(t,x)\big)- e^{-\lm t}\bar{v}_0- Mt
\end{equation}
where $M>0$ has been defined in \eqref{H2},  $\lambda>0$ to be fixed and    $\bar{v}_0$ (the upper bound of $v_0$) is assumed, without loss of generality, to be positive
\begin{equation}\label{diff}
\begin{array}{c}
\displaystyle
\intO \zne(t)\vp(t)\,dx+\iint_{Q_t}\left(\lm \zne+ M\right)\vp\,dx\,ds+\iint_{Q_t}\frac{e^{-\lm t}}{\eps}A(x)\left(|\N  u_n|^{p-2}\N  u_n
-|\N  v^\eps_n|^{p-2}\N  v^\eps_n \right)\cdot\N\vp\,dx\,ds
\\[4mm]
\displaystyle \le\iint_{Q_t} \frac{e^{-\lm t}}{\eps}\biggl[H\left(s,x,\N u_n\right)-(1-\eps)^{p-1}H\left((1-\eps)^{p-2}s,x,\frac{\N v_n^\eps}{1-\eps}\right)\biggr]\vp\,dx\,ds+\intO \zne(0)\vp(0)\,dx+R_n
\end{array}
\end{equation}
and
\[
\begin{split}
R_n&=-\iint_{Q_t}\frac{e^{-\lm t}}{\eps}A(x)\big[\tenu |\N u|^{p-2}\N u-|\N \un|^{p-2}\N \un\big]\cdot
\N\vp\,dx\,ds\\
&\q  -\iint_{Q_t}\te_n'(u)A(x)|\N u|^{p-2}\N u\cdot \N u\,\vp\,dx\,ds +\iint_{Q_t}\big[H(s,x,\N u)\tenu-H(s,x,\N \un)\big]\vp\,dx\,ds\\
&\q
-\iint_{Q_t}\frac{e^{-\lm t}}{\eps}A(x)\big[
\te_n(v((1-\eps)^{p-2}t,x)) |\N \ve|^{p-2}\N \ve-|\N \vne|^{p-2}\N \vne
\big]\cdot\N\vp\,dx\,ds\\
&\q
-(1-\eps)^{-1}\iint_{Q_t}\te_n'(v((1-\eps)^{p-2}t,x))A(x)|\N  \ve|^{p-2}\N \ve\cdot \N v\,\vp\,dx\,ds\\
&\q
 +\iint_{Q_t}\big[H((1-\eps)^{p-2}s,x,\N \ve)\te_n(v((1-\eps)^{p-2}t,x))-H((1-\eps)^{p-2}s,x,\N \vne)\big]\vp\,dx\,ds.
\end{split}
\]
%In particular, the assumption \eqref{H2} provides us with
%\begin{equation}\label{eq}
%\begin{array}{c}
%\displaystyle
%H(s,x,\N u_n)-(1-\eps)^{p-1}H\left((1-\eps)^{p-2}s,x,\frac{\N v_n^\eps}{1-\eps}\right)
%\le c_1\eps e^{\lm q t}|\N \zne|^q+c_2|\N \zne|\left[
%|\N u_n|^{\frac{p-2}{2}}+|\N v^\eps_n|^{\frac{p-2}{2}}
%\right]
%+\eps M.
% \end{array}
%\end{equation}

Let us define, for any $\mu>0$ and $a\ge 1$, the following function:  
\begin{equation}\label{psimu}
\Psi_{a,\mu}(v)=\int_0^{ G_k(v) } (\mu+|w|)^{\frac{\sigma-2}{a}}\,dw.
\end{equation}
Then, we set 
$$\ds \vp(\zne)= \Psi_{1,\mu}(\zne)=\int_0^{ G_k({z}_n^\eps) }(\mu+|w|)^{\si-2}\,dw $$ 
in  \eqref{diff}, with $k\ge \bar{v}_0$ (in order to have $G_k (z_n^\eps (0,x)) \equiv 0$ in $\Omega$)   in \eqref{psimu} so that, thanks to \eqref{H2}, we get
\begin{equation}\label{1}
\begin{array}{cr}
\ds
\int_{\Omega} \Vp(\zne(t))\,dx+\lm \iint_{Q_t} G_k({z}_n^\eps) \left(\int_0^{ G_k({z}_n^\eps) }(\mu+|w|)^{\si-2}\,dw\right)\,dx\,ds
+\frac{\alpha }{2}\eps^{p-2}\iint_{Q_t}|\N \Psi_{p,\mu}( G_k({z}_n^\eps) )|^p \,dx\,ds\\
[4mm]\ds
+\frac{\alpha }{2}\iint_{Q_t}|\N \Psi_{2,\mu}( G_k({z}_n^\eps) )|^2\left[
|\N {u}_n|^{p-2}+|\N {v}_n^\eps|^{p-2}
\right] \,dx\,ds
\\
[4mm]\ds
\le c_1e^{\lm(q-1)T}\iint_{Q_t}|\N  G_k(\zne) |^q\left(\int_0^{ G_k({z}_n^\eps) }(\mu+|w|)^{\si-2}\,dw\right)\,dx\,ds& \\
[4mm]
\ds
+c_2e^{\frac{p-2}{2}\lm T}\iint_{Q_t}|\N  G_k(\zne) |\left[
|\N {u}_n|^{\frac{p-2}{2}}+|\N {v}_n^\eps|^{\frac{p-2}{2}}
\right]\left(\int_0^{ G_k(\hat{z}_n^\eps) }(\mu+|w|)^{\si-2}\,dw\right)\,dx\,ds 
+\Rn \,.&
\end{array}
\end{equation}

The proof $\ds \lim_{n\to\infty}\Rn=0$ follows using the analogous one contained in \cref{teorin2} with $\ell=0$ (see \eqref{A}), changing \eqref{potc} with \eqref{pot} and taking advantage of the definition of $\vp(\cdot)$.%\footnote{frase criptica...\\ modificata}

We first deal with the latter integral in the right hand side.   The definition of $|\N \Psi_{2,\mu} ( G_k({z}_n^\eps) )|$ and Young's inequality yield to 
\begin{equation*}% \label{b-uno}
\begin{array}{c}
\ds
c_2e^{\frac{p-2}{2}\lm T}\iint_{Q_t}|\N   G_k(\zne)  |\left[
|\N {u}_n|^{\frac{p-2}{2}}+|\N {v}_n^\eps|^{\frac{p-2}{2}}
\right]\left(\int_0^{ G_k(\zne) }(\mu+|w|)^{\si-2}\,dw\right)\,dx\,ds\\
[4mm]\ds
\le c_2 e^{\frac{p-2}{2}\lm T}\iint_{Q_t}|\N \Psi_{2,\mu}( G_k(\zne) )|\left[
|\N {u}_n|^{\frac{p-2}{2}}+|\N {v}_n^\eps|^{\frac{p-2}{2}}
\right]\Psi_{2,\mu}( G_k(\zne) )\,dx\,ds\\
[4mm] 
\ds
\le \frac{\alpha }{2}\iint_{Q_t}|\N \Psi_{2,\mu}( G_k(\zne) )|^2\left[
|\N {u}_n|^{\frac{p-2}{2}}+|\N {v}_n^\eps|^{\frac{p-2}{2}}
\right]\,dx\,ds+c\iint_{Q_t}\left(\Psi_{2,\mu}( G_k(\zne) )\right)^2\,dx\,ds\\
[4mm]\ds
\le \frac{\alpha }{2}\iint_{Q_t}|\N \Psi_{2,\mu}( G_k({z}_n^\eps) )|^2\left[
|\N {u}_n|^{\frac{p-2}{2}}+|\N \hat{v}_n^\eps|^{\frac{p-2}{2}}
\right]\,dx\,ds
%\\
%[4mm]
%\ds
+\bar{c}\iint_{Q_t} G_k(\zne) \left(\int_0^{ G_k(\zne) }(\mu+|w|)^{\si-2}\,dw\right)\,dx\,ds.
\end{array}
\end{equation*}
Then, we use such a estimate to prove that, having also \eqref{pot}, provides us with the uniform bound in $n$ of $(1+ G_k(\zne) )^{\gamma-1} G_k(\zne) $ in $ L^p(0,T;W^{1,p}_0(\Omega))$. Indeed, \eqref{1} becomes
\begin{equation*} %\label{1}
\begin{array}{c}
\ds
\int_{\Omega} \Vp(\zne(t))\,dx
+\frac{\alpha }{2}\eps^{p-2}\iint_{Q_t}|\N \Psi_{p,\mu}( G_k(\zne) )|^p \,dx\,ds
\\
[4mm]\ds
\le c_1e^{(q-1)T}\iint_{Q_t}|\N  G_k(\zne) |^q\left(\int_0^{ G_k(\zne) }(\mu+|w|)^{\si-2}\,dw\right)\,dx\,ds
+\omn
\end{array}
\end{equation*}
provided $\lm\ge \bar{c}$, and since
\begin{equation}\label{AAA}
\begin{array}{c}
\ds
\iint_{Q_t}|\N  G_k(\zne) |^q\left(\int_0^{ G_k(\zne) }(\mu+|w|)^{\si-2}\,dw\right)\,dx\,ds\\
[4mm]\ds
\le \iint_{Q_t}|\N \Psi_{p,\mu}( G_k(\zne) )|^q\biggl(\int_0^{ G_k(\zne) }(\mu +|w|)^{(\sigma-2)\frac{p-q}{p}}\,dw\biggr)\,dx\,ds\\
[4mm]\ds
\le 
\iint_{Q_t} |\N \Psi_{p,\mu}( G_k(\zne) )|^q 
 |\Psi_{p,\mu}( G_k(\zne) )|^{p-q} G_k(\zne) ^{q-(p-1)}\,dx\,ds
\end{array}
\end{equation}
by definition of $\Psi_{p,\mu}(\cdot)$ and H\"older's inequality with indices $\left(\frac{1}{p-q},\frac{1}{q-(p-1)}\right)$, we have that
\begin{equation}\label{2}
\begin{array}{cr}
\ds
\int_{\Omega} \Vp(\zne(t))\,dx
+\frac{\alpha }{2}\eps^{p-2}\iint_{Q_t}|\N \Psi_{p,\mu}( G_k(\zne) )|^p \,dx\,ds
\\
[4mm]\ds
\le c_1\iint_{Q_t} |\N \Psi_{p,\mu}( G_k(\zne) )|^q 
 |\Psi_{p,\mu}( G_k(\zne) )|^{p-q} G_k(\zne) ^{q-(p-1)}\,dx\,ds+\omn. &
\end{array}
\end{equation}
An application of Young's inequality with $\left(\frac{p}{q},\frac{p}{p-q}\right)$  implies
\begin{equation*} %\label{1}
\begin{array}{cr}
\ds
\int_{\Omega} \Vp(\zne(t))\,dx
+\frac{\alpha }{4}\eps^{p-2}\iint_{Q_t}|\N \Psi_{p,\mu}( G_k(\zne) )|^p \,dx\,ds
\le c\iint_{Q_t}  
 |\Psi_{p,\mu}( G_k(\zne) )|^{p} G_k(\zne) ^{\frac{p(q-(p-1))}{p-q}}\,dx\,ds
+\omn. &
\end{array}
\end{equation*}
The uniform boundedness of the right hand side  above is due to the fact that $p\ga+\frac{p(q-(p-1))}{p-q}=p\frac{N\b+\si}{N}$ and that $(u-v)\in L^{p\frac{N+\si}{N}}(Q_T)$ by \eqref{pot}.\\
We continue applying once more  the H\"older inequality with indices $\left(\frac{p}{q},\frac{p^*}{p-q},\frac{N}{p-q}\right)$ and also Sobolev's embedding, so we finally get
\begin{equation*}%\label{a-uno}
\begin{array}{c}
\ds
\iint_{Q_t} |\N \Psi_{p,\mu}( G_k(\zne) )|^q 
 |\Psi_{p,\mu}( G_k(\zne) )|^{p-q} G_k(\zne) ^{q-(p-1)}\,dx\,ds
 \\
 [4mm]\ds
 \le c \sup_{s\in (0,t)}\| G_k(\zne(s)) \|_{L^{\sigma}(\Omega)}^{q-p+1} \int_0^t   \|\N \Psi_{p,\mu} ( G_k(\zne(s)) )\|_{L^p(\Omega)}^p \,ds
\end{array}
\end{equation*}
where $c=c(\bar \ga, N,q,T)$ and finally deduce that
\begin{equation*}
\begin{array}{c}
\ds
\int_{\Omega} \Vp(\zne(t))\,dx
+\frac{\alpha}{2} \eps^{p-2}\iint_{Q_t}|\N \Psi_{p,\mu}( G_k(\zne) )|^p \,dx\,ds
\\
[4mm]\ds
\le c \sup_{s\in (0,t)}\| G_k(\zne(s)) \|_{L^{\sigma}(\Omega)}^{q-p+1} \int_0^t   \|\N \Psi_{p,\mu} ( G_k(\zne(s)) )\|_{L^p(\Omega)}^p \,ds
+\omn.
\end{array}
\end{equation*}
Then, the above uniform boundedness in $n$ on the difference between sub/supersolutions allow us to let $n\to\infty$ getting
\begin{equation*}
\begin{array}{c}
\ds
\int_{\Omega} \Vp(z^\eps(t))\,dx
+\frac{\alpha }{2}\eps^{p-2}\iint_{Q_t}|\N \Psi_{p,\mu}( G_k({z}^\eps) )|^p \,dx\,ds
\le c\sup_{s\in (0,t)}\| G_k({z}^\eps(s)) \|_{L^{\sigma}(\Omega)}^{q-p+1} \iint_{Q_t}|\N \Psi_{p,\mu}( G_k({z}^\eps) )|^p \,dx\,ds.
\end{array}
\end{equation*}
We now reason as in \cref{comp2} and,  being $\Vp(w) \longrightarrow \frac{|w|^\sigma}{\sigma(\sigma-1)}$ as $\mu\to 0$ and thanks to \cref{lemell}, we have that $G_k ({z}^\eps ) \equiv 0$ in $Q_t$. In particular, this means that
\begin{equation*}
e^{-\lm t}\bigg(u\big(t,x\big)-(1-\eps)v\big((1-\eps)^{p-2}t,x\big) -\eps \bar{v}_0\bigg)-\eps Mt\le \eps k
\end{equation*}
and letting $\varepsilon$ vanishes we deduce that $u (t,x)\le v(t,x)$ in $Q_T$, as desired.

\endproof

\begin{proof}[Proof of \cref{teorinL1}] $\,$ 
We set $\lm=1$ in \eqref{zneh2} and take
$$ \vp(\zne)=1-\frac{1}{(1+ G_k(\zne) )^\mu},\quad \mbox{with} \quad \mu=(p-q)\frac{N+1}{N}-1>p-1
$$
 in \eqref{diff}, so that 
\begin{equation}\label{in5}
\begin{array}{c}
\ds
\integrale \Vp(\zne(t))\,dx
+\frac{\mu\al}{2}\eps^{p-2}\iint_{Q_t}\frac{|\N   G_k(\zne)  |^p}{(1+ G_k(\zne) )^{\mu+1}}\,dx\,ds
\\
[4mm]
\ds
+\frac{\mu\al}{2}\iint_{Q_t}\frac{|\N   G_k(\zne)  |^2}{(1+ G_k(\zne) )^{\mu+1}}\left[
|\N {u}_n|^{p-2}+|\N {v}^\eps_n|^{p-2}
\right]
\,dx\,ds
\\
[4mm]
\ds
\le c_1e^{(q-1)\lm T}\iint_{Q_t}|\N  G_k(\zne) |^q\left[1-\frac{1}{(1+ G_k(\zne) )^\mu}\right]\,dx\,ds \\
[4mm]\ds
+c_2e^{\frac{p-2}{2}\lm T}\iint_{Q_t}|\N ( G_k(\zne) )|\left[
|\N {u}_n|^{\frac{p-2}{2}}+|\N {v}^\eps_n|^{\frac{p-2}{2}}\right]\left[1-\frac{1}{(1+ G_k(\zne) )^\mu}\right]\,dx\,ds
+\Rn
\end{array}
\end{equation}
thanks to \eqref{H2}.

The proof that  $\ds \lim_{n\to\infty}\Rn=0$   follows reasoning as in \cref{teo33} (see also \cref{teorinL12}), so we skip it. 

\medskip

We start estimating the first integral in the right hand side above through Young's inequality with indices $\left(\frac{p}{q},\frac{p}{p-q}\right)$, getting
\[
\begin{array}{c}
\ds
c_1e^{(q-1)\lm T}\iint_{Q_t}|\N  G_k(\zne) |^q\left[1-\frac{1}{(1+ G_k(\zne) )^\mu}\right]\,dx\,ds\\
[4mm]\ds
\le \frac{\mu\al}{4}\eps^{p-2}\iint_{Q_t}\frac{|\N ( G_k(\zne) )|^p}{(1+ G_k(\zne) )^{\mu+1}}\,dx\,ds+c\iint_{Q_t}(1+ G_k(\zne) )^{\frac{p(\mu+1)}{p-q}}\,dx\,ds.
\end{array}
\]
As far as the second integral in the right hand side  of \eqref{in5}, applying again Young's inequality with indices $(2,2)$ we get 
\begin{equation}\label{4}
\begin{array}{c}
\ds
c_2e^{\frac{p-2}{2}\lm T}\iint_{Q_t}|\N   G_k(\zne)  |\left[
|\N {u}_n|^{\frac{p-2}{2}}+|\N {v}^\eps_n|^{\frac{p-2}{2}}\right]\left[1-\frac{1}{(1+ G_k(\zne) )^\mu}\right]\,dx\,ds\\
[4mm]\ds
\le \frac{\mu}{2}\iint_{Q_t}\frac{|\N   G_k(\zne)  |^2}{(1+ G_k(\zne) )^{\mu+1}}\left[
|\N \un|^{p-2}+|\N \vn|^{p-2}\right]\,dx\,ds
+c\iint_{Q_t}(1+ G_k(\zne) )^{\mu+1}\,dx\,ds.
\end{array}
\end{equation}
%\footnote{ $\frac{\mu}{2}\eps^{p-2}\iint_{Q_t}\frac{|\N ( G_k(\zne) )|^2}{(1+ G_k(\zne) )^{\mu+1}}\left[ |\N \hat{u}_n|^{p-2}+|\N \hat{v}^\eps_n|^{p-2}\right]\,dx\,ds +c\iint_{Q_t}(1+ G_k(\zne) )^{\mu+1}\,dx\,ds.$ sfuggito a modifiche precedenti, giusto?}
Using the previous estimates into \eqref{in5} we obtain
\begin{equation*}
\begin{array}{c}
\ds
\integrale \Vp(\zne(t))\,dx
+\frac{\mu\al}{4}\eps^{p-2}\iint_{Q_t}\frac{|\N ( G_k(\zne) )|^p}{(1+ G_k(\zne) )^{\mu+1}}\,dx\,ds
\\
[4mm]
\ds
\le c\left[\iint_{Q_t}(1+ G_k(\zne) )^{\frac{p(\mu+1)}{p-q}}\,dx\,ds
+\iint_{Q_t}(1+ G_k(\zne) )^{\mu+1}\,dx\,ds\right]+\omn.
\end{array}
\end{equation*}
Then, since having $p>2$ and $q>\frac{p(N+1)-N}{N+2}$ imply that both $\frac{p(\mu+1)}{p-q}$ are $\mu+1 $ smaller than $q\frac{N+1}{N}$ (see \cref{lemmaq} and \cref{teoGN}), the right hand side  in the above inequality is uniformly bounded with respect to  $n$.

\medskip 

Now, let us focus on the right hand side  of \eqref{in5}. We apply H\"older's inequality with $\left(\frac{p}{q},\frac{p}{p-q}\right)$ on the integral involving the $q$ power of the gradient, getting
\begin{equation*}
\begin{array}{c}
\ds
c_1e^{(q-1)\lm T}\iint_{Q_t}|\N  G_k(\zne) |^q\left[1-\frac{1}{(1+ G_k(\zne) )^\mu}\right]\,dx\,ds\\
[4mm]
\ds
\le c\left(\iint_{Q_t}(1+ G_k(\zne) )^{\frac{q(\mu+1)}{p-q}}\left[1-\frac{1}{(1+ G_k(\zne) )^\mu}\right]^{\frac{p}{p-q}}\,dx\,ds\right)^{\frac{p-q}{p}}\left(\iint_{Q_t}\frac{|\N  G_k(\zne) |^p}{(1+ G_k(\zne) )^{\mu+1}}\,dx\,ds\right)^{\frac{q}{p}}.
\end{array}
\end{equation*}
We   observe that the value $\nu=p\frac{N+\frac{p}{p-\mu-1}}{N}$ corresponds to the Gagliardo-Nirenberg regularity exponent for the spaces
\[
L^\infty(0,T;L^{\frac{p}{p-\mu-1}}(\Omega))\cap L^p(0,T;W^{1,p}_0(\Omega)).
\]
We justify such a gradient regularity recalling \cref{lemmaq}.
Then, we have
\begin{equation*}
\begin{split}
\iint_{Q_T}(1+ G_k(\zne) )^{\frac{q(\mu+1)}{p-q}}\left[1-\frac{1}{(1+ G_k(\zne) )^\mu}\right]^{\frac{p}{p-q}}\,dx\,dt
&\le 
\int_0^T\left\|(1+ G_k(\zne) )^{-\frac{\mu+1}{p}} G_k(\zne) \right\|_{L^{\nu}(\Omega)}^{\nu}\,dt\\
&\le {c_{GN}} \| G_k(\zne) \|_{L^\infty(0,T;L^{1}(\Omega))}^{\frac{p}{N}}\iint_{Q_T}\frac{|\N  G_k(\zne) |^p}{(1+ G_k(\zne) )^{\mu+1}}\,dx\,dt
\end{split}
\end{equation*}
thanks to \eqref{disGN}. Furthermore, since the last integral in \eqref{4} can be estimated by
\begin{equation*}
\begin{split} 
\ds
\iint_{Q_t}(1+ G_k(\zne) )^{\mu} G_k(\zne) \,dx\,ds
&\le \iint_{Q_t}(1+ G_k(\zne) )^{q\frac{N+1}{N}-1} G_k(\zne) \,dx\,ds\\
&\le c\| G_k(\zne) \|_{L^\infty(0,T;L^{1}(\Omega))}^{\frac{p-q}{N}} \iint_{Q_T}\frac{|\N  G_k(\zne) |^p}{(1+ G_k(\zne) )^{\mu+1}}\,dx\,dt 
\end{split}
\end{equation*}
we are left with the study of
\begin{equation*}
\begin{array}{c}
\ds
\integrale \Vp(\zne(t))\,dx
+\frac{\mu\al}{2}\eps^{p-2}\iint_{Q_t}\frac{|\N ( G_k(\zne) )|^p}{(1+ G_k(\zne) )^{\mu+1}}\,dx\,ds
\le c\| G_k(\zne) \|_{L^\infty(0,T;L^{1}(\Omega))}^{\frac{p-q}{N}}  \iint_{Q_T}\frac{|\N  G_k(\zne) |^p}{(1+ G_k(\zne) )^{\mu+1}}\,dx\,dt +\omn.
\end{array}
\end{equation*}
Then, letting $n\to \infty$ in the inequality above yields to
\begin{equation*}
\begin{array}{c}
\ds
\integrale \Vp(z^\eps(t))\,dx
+\frac{\mu\al}{2}\eps^{p-2}\iint_{Q_t}\frac{|\N ( G_k({z}^\eps) )|^p}{(1+ G_k({z}^\eps) )^{\mu+1}}\,dx\,ds \le c\| G_k(z^\eps) \|_{L^\infty(0,t;L^{1}(\Omega))}^{\frac{p-q}{N}}\iint_{Q_t}\frac{|\N  G_k(z^\eps) |^p}{(1+ G_k(z^\eps) )^{\mu+1}}\,dx\,ds
\end{array}
\end{equation*}
We conclude reasoning as in the proof of \cref{teorinL12}, recalling \cref{lemell} and letting $\eps\to 0$.

\end{proof}

\appendix 

\section{}%On well posedness class %\eqref{potc} inserire dopo aver tolto refcheck
\label{app}

Our current goal is proving that one needs to consider sub/supersolutions belonging to the regularity class \eqref{csi}--\eqref{potc} in order to have a uniqueness result for problems of \eqref{model} type.

Here we use a result contained in  \cite[Section $3$]{BASW1} (see also \cite{BASW2}),   where it is proved that the Cauchy problem
\begin{equation}\label{p2q}
\begin{cases}
\begin{array}{ll}
u_t-\D u=c_1  |\N u|^q & \text{in}\,\,(0,T)\times\R^N,\\
u(0,x)=0  &\text{in} \,\,\R^N,
\end{array}
\end{cases}
\end{equation}with $c_1>0$ and $q>1$ and $N\geq 2$
admits at least two  solutions. 

More precisely they prove that there exists a value $\alpha_0>0$ and a (unique) solution    $U\in C^2(0,\infty)\cap C^1[0,\infty)$  of the following Cauchy problem: 
\begin{equation}\label{aeiouy}
\begin{cases}
\begin{array}{l}	\d
U''+\left( \frac{N-1}{y}+\frac{y}{2}  \right)U'+kU+c_1  |U'|^q=0\q \text{for}\,\,0<y<\infty,\\
U'(0)=0,\\
U(0)=\al_0
\end{array}
\end{cases}
\end{equation}
 that   satisfies 
 \begin{equation}\label{Ur}
\d U(y)=ce^{-\frac{y^2}{4}}y^{-\frac{N}{\si'}}\left(1+o(y^{-2}) \right) \q\t{as}\q y\to\infty, \quad \mbox{with } \sigma=\frac{N(q-1)}{2-q} , 
\end{equation}
\begin{equation}\label{Upr}
\d U'(y)=-\frac{y}{2}U(y)\left(1+o(1) \right) \q\t{as}\q y\to\infty
\end{equation}
and 
\begin{equation}\label{limU}
U\in L^j([0,\infty);y^{N-1}dy)\q\t{for any }\q 1\le j<\infty. 
\end{equation}

Consequently it is easy to see that both  $u_1\equiv 0$ and  $u_2(t,x)=t^{-\frac{N}{2\si}}U(|x|/\sqrt{t})$ solve \eqref{p2q}.

\medskip 

%\footnote{\hl{non trovo il lavoro (credo di averlo sull'altro computer); va detto qualcosa su $\alpha$ altrimenti qui staremmo trovando infinite soluzioni...\\
%ho aggiunto la definizione di $\al$ ma va detto qualcosa in più sul perché quel valore va bene}}

We  use such a  result in order to show that the class of uniqueness \eqref{csi}--\eqref{potc} in the right one in order to have comparison (and thus uniqueness). Indeed we construct, for the following problem 

 \begin{equation*}%\label{model2}%%%%\tag{$\D_p$}
\begin{cases}
 \ds u_t- \Delta  u=|\N u|^q + f(t,x) &\ds\text{in}\quad (0,T)\times  B_R (0) ,\\
 \ds u (t,x) =0 &\ds\text{on}\quad(0,T)\times \partial B_R (0),\\
 \ds u(0,x)=0 &\ds\text{in}\quad B_R (0),
\end{cases}
\end{equation*}
for  a suitable choice of   $f$ smooth,  a pair of solutions, whose only one belong to the class  \eqref{csi}--\eqref{potc}, while  the other one   is not regular enough.
\bigskip 

In fact we have the following result. 
\begin{theorem}
Let $2-\frac{N}{N+1}<q<2$, $R>0$ and let $U$ be the positive solution of \eqref{aeiouy}. Then  the Cauchy-Dirichlet problem
\begin{equation}\label{Pv}
\begin{cases}
\begin{array}{ll}
v_t-\D v=c_1 |\N v|^q+\left(t^{-\frac{N}{2\si}}U(R/\sqrt{t})\right)'  & \text{in}\,\,(0,T)\times B_R(0),\\
v=0 &\text{on}\,\,(0,T)\times \partial B_R(0) ,\\
v(0,x)= 0 & \text{in}\,\,B_R(0),
\end{array}
\end{cases}
\end{equation}
admits at least two solutions $v_{1,2}$ such that: 
\begin{itemize}
\item $ v_1\in C([0,T];L^\infty( B_R(0)))$ and $ |v_1|^{\frac{\si}{2}} \in L^2(0,T;H_0^1(B_R(0)))
$, with $\sigma= \frac{N(q-1)}{2-q}$; 
\medskip 

\item $v_2(t,x)=t^{-\frac{N}{2\si}}\big(U(|x|/\sqrt{t})-U(R/\sqrt{t})\big),$ 
%where $U$ solves  \eqref{aeiouy},  
and it satisfies  
\begin{equation}\label{nonu2}
 v_2\in  C([0,T];L^\mu( B_R(0)))\q\t{for any}\q 1\le \mu<\si\q\t{but}\q  v_2\notin  C([0,T];L^\si( B_R(0)))
\end{equation}
and
\begin{equation}\label{nonu3}
|v_2|^{\b} \in L^2(0,T;H_0^1(B_R(0))),\q\t{for any}\q \b<\frac{\si}{2},\q\t{but}\q |v_2|^{\frac{\si}{2}} \notin L^2(0,T;H_0^1(B_R(0))).
\end{equation}
\end{itemize} 
\end{theorem}

\begin{proof}
We proceed observing  that, thanks to \eqref{Ur}--\eqref{Upr}, then  $\left(t^{-\frac{N}{2\si}}U\left(\frac{R}{\sqrt{t}}\right)\right)'\in C^1([0,T])$ and thus    \eqref{Pv} admits a solution $v_1$ such that $ v_1\in C([0,T];L^\sigma (\Omega))$ and $ |v_1|^{\frac{\si}{2}} \in L^2(0,T;H_0^1(B_R(0)))$ 
 (see \cite{Ma}).
 
 \medskip 
 
Then, we are left with the proofs of \eqref{nonu2}--\eqref{nonu3}.
\medskip

%\noindent
%\textit{Continuity in $L^{\mu}( B_R(0))$: \eqref{nonu2}.}
%\noindent

In order  to prove that $ \|v_2(t)\|_{L^\mu(B_R(0))}\to 0$, as $t\to 0^+$, for $\mu<\si$,  we compute
\[
\begin{array}{c}\ds
\int_{B_R(0)} |v_2(t,x)|^\mu\,dx=
t^{-\mu \frac{N}{2\si}}\int_0^R |U(r/\sqrt{t})-U(R/\sqrt{t}) |^\mu r^{N-1}\,dr
\ds=
t^{\frac{N}{2}\left(1-\frac{\mu}{\si}\right)}\int_0^{\frac{R}{\sqrt{t}}} |U(y)-U({R}/{\sqrt{t}})|^\mu y^{N-1}\,dy\\
[4mm]\ds\le
t^{\frac{N}{2}\left(1-\frac{\mu}{\si}\right)}\int_0^\infty |U(y)-U({R}/{\sqrt{t}})|^\mu y^{N-1}\,dy\le ct^{\frac{N}{2}\left(1-\frac{\mu}{\si}\right)}
\end{array}
\]
where the last inequality follows from \eqref{limU}.
%proved in \cite{BASW2}. 
Then, if $\mu<\si$, we have  $\frac{N}{2}\left(1-\frac{\mu}{\si}\right)>0$ which implies that the right hand side  of the above inequality  vanishes as $t\to 0^+$. 
If, on the contrary, we set $\mu=\si$ then the integral above becomes
\[
\begin{split}
\int_{B_R(0)} |v_2(t,x)|^\si\,dx&=
\int_0^{\frac{R}{\sqrt{t}}} |U(y)-U({R}/{\sqrt{t}})|^\si y^{N-1}\,dy
\end{split}
\]
which is bounded from below, thanks to \eqref{limU},  by a positive constant. 

 In order to prove \eqref{nonu3}, we observe that 
\begin{equation}\label{der}
\begin{array}{c}\ds
\int_0^T\int_{B_R(0)}|\N |v_2|^\b|^2\,dx\,dt=\int_0^T\int_0^R|\N |t^{-\frac{N}{2\si}}U(r/\sqrt{t})|^\b|^2r^{N-1}\,dr\,dt\\
[4mm] \ds
=\b^2\int_0^T\int_0^Rt^{-\frac{N}{\si}\b-1}U(r/\sqrt{t})^{2(\b-1)}U'(r/\sqrt{t})^2r^{N-1}\,dr\,dt \\ \ds
=\b^2\int_0^Tt^{-\frac{N}{\si}\b-1+\frac{N}{2}} \int_0^{\frac{R}{\sqrt{t}}} U(y)^{2(\b-1)}U'(y)^2y^{N-1}\,dy\,dt.
\end{array}
\end{equation}
Then, recalling  \eqref{Ur}--\eqref{Upr}, we get 
\[
\begin{array}{c}\ds
\int_0^T\int_{B_R(0)}|\N |v_2|^\b|^2\,dx\,dt
\le c\int_0^T t^{-\frac{N}{\si}\b-1+\frac{N}{2}} \int_0^{\infty}U(y)^{2(\b-1)}U'(y)^2y^{N-1}\,dy\,dt\\
[4mm]
\ds \le c\int_0^Tt^{-\frac{N\b}{\si}-1+\frac{N}{2}}\int_0^\infty e^{-\b\frac{y^2}{2}}y^{-2N\b(\frac{\si-1}{\si})+N+1}\,dy\,dt
\le c \int_0^Tt^{-\frac{N\b}{\si}-1+\frac{N}{2}}\,dt\,,
\end{array}
\]
which is finite if $\b<\frac{\si}{2}$.
On the other hand, if we  suppose that $\b=\frac{\si}{2}$, then \eqref{der} becomes
\[
\begin{array}{c}\ds
\int_0^T\int_{B_R(0)}|\N |v_2|^{\frac{\si}{2}}|^2\,dx\,dt=\frac{\si^2}{4}\int_0^Tt^{-1}\int_0^{\frac{R}{\sqrt{t}}}U(y)^{\si-2}U'(y)^2y^{N-1}\,dy\,dt\\
[4mm]\ds \ge \frac{\si^2}{4}\int_0^1t^{-1}\int_0^RU(y)^{\si-2}U'(y)^2y^{N-1}\,dy\,dt 
\ge c\int_0^1t^{-1}\,dt = +\infty \,,
\end{array}
\]
and the assertion follows. \end{proof}

\end{document}

%% file: style.tex
\usepackage{geometry}
\usepackage[latin1]{inputenc}

\usepackage[scaled=.95]{helvet} 
\usepackage{courier}
\usepackage{amssymb}
\usepackage{amsmath}
\usepackage{amsthm}
\usepackage{amscd}
\usepackage{amssymb}
\usepackage{amsfonts}
\usepackage{graphics}
\usepackage[english]{babel}
\usepackage{bbm}
\usepackage{scalerel}
\usepackage{stackengine}
\usepackage{chngcntr}                   % allows changing equation section depth mid-document
\usepackage{color,soul}

\usepackage{hyperref}
\hypersetup{colorlinks=true}
\hypersetup{colorlinks=true,
linkcolor=red!70!black,
citecolor=orange,
urlcolor=blue}

\usepackage[capitalize,nameinlink,noabbrev,nosort]{cleveref} % to emulate \autoref style

\crefname{section}{Section}{Sections}
\crefname{subsection}{Subsection}{Subsections}

\crefname{figure}{Figure}{Figures}

\crefname{definition}{Definition}{Definitions}
\crefname{theorem}{Theorem}{Theorems}
\crefname{proposition}{Proposition}{Propositions}
\crefname{corollary}{Corollary}{Corollaries}
\crefname{remark}{Remark}{Remarks}
\crefname{lemma}{Lemma}{Lemmata}
\crefname{remarks}{Remarks}{Remarks}

\usepackage{graphics}
\usepackage{pgfplots}
\pgfplotsset{width=7cm,compat=newest} 
\usepackage{graphics}
\usepackage{tikz}
\usetikzlibrary{shapes.geometric}
\usepackage{tikz-cd}
\usepackage{caption}
\usepackage{subcaption}

\usepackage{float}

\stackMath
\def\hatgap{2pt}
\def\subdown{-2pt}
\newcommand\reallywidehat[2][]{%
\renewcommand\stackalignment{l}%
\stackon[\hatgap]{#2}{%
\stretchto{%
    \scalerel*[\widthof{$#2$}]{\kern-.6pt\bigwedge\kern-.6pt}%
    {\rule[-\textheight/2]{1ex}{\textheight}}%WIDTH-LIMITED BIG WEDGE
}{0.5ex}% THIS SQUEEZES THE WEDGE TO 0.5ex HEIGHT
_{\smash{\belowbaseline[\subdown]{\scriptstyle#1}}}%
}}

\renewcommand{\theequation}{\thesection.\arabic{equation}}
\numberwithin{equation}{section}
\newtheoremstyle{mytheoremstyle} % name
    {.3cm}                    % Space above
    {.3cm}                    % Space below
    {\itshape}                   % Body font
    {}                           % Indent amount
    {\bfseries}             % Theorem head font
    {.}                          % Punctuation after theorem head
    {.5em}                       % Space after theorem head
    {}  % Theorem head spec (can be left empty, meaning ?normal?)

\theoremstyle{mytheoremstyle}

\newtheorem{theorem}{Theorem}[section]

\newtheorem{lemma}[theorem]{Lemma}

\newtheorem{definition}[theorem]{Definition}

\newtheorem{remark}[theorem]{Remark}

\def\og{\leavevmode\raise.3ex\hbox{$\scriptscriptstyle\langle\!\langle$~}}
\def\fg{\leavevmode\raise.3ex\hbox{~$\!\scriptscriptstyle\,\rangle\!\rangle$}}

\def\q{\quad}
\def\qq{\qquad}
\def\t{\text}
\def\omn{\om_n }
\def\Rn{R_n }
\def\intO{\displaystyle \int_{\Omega}}
\def\eqref#1{(\ref{#1})}

\def\rn{\mathbb{R}^N}
\def\R{\mathbb{R}}

\def\vp{\varphi}
\def\lm{\lambda}
\def\al{\alpha}
\newcommand{\D}{\Delta}
\newcommand{\si}{\sigma}

\newcommand{\integrale}{\int_{\Omega}}
\def\eps{\varepsilon}

\def\om{\omega}
\def\de{\delta}
\def\dive{\text{ div }}
\def\dys{\displaystyle}

\def\bc{\begin{cases}}
\def\ec{\end{cases}}
\def\be{\begin{equation}}
\def\ee{\end{equation}}
\def\theequation{\arabic{section}.\arabic{equation}}

\newcommand{\ds}{\displaystyle}
\def\te{\theta}
\def\un{u_n}

\def\vn{v_n}

\def\vne{v_n^\eps}
\def\ve{v^\eps}

\def\zn{z_n}
\def\zne{z_n^\eps}

\def\tenu{\te_n(u)}
\def\up{u_+}
\def\vm{v_-}
\def\tenv{\te_n(v)}
\def\kin{\chi_{\{\un>\vn\}}}
\def\ki{\chi_{\{u>v\}}}

\def\N{\nabla}

\def\B{\mathcal{B}}
\def\Vp{\Phi}
\def\b{\beta}
\def\ga{\gamma}
\def\vp{\varphi}
\def\al{\alpha}
\def\om{\omega}

\def\eps{\varepsilon}

\def\ro{\rho}
\def\te{\theta}
\def\un{u_n}

\def\vn{v_n}

\def\vne{v_n^\eps}
\def\ve{v^\eps}

\def\zn{z_n}
\def\zne{z_n^\eps}

\def\tenu{\te_n(u)}
\def\tenv{\te_n(v)}
\def\kin{\chi_{\{\un>\vn\}}}
\def\ki{\chi_{\{u>v\}}}

\newcommand{\car}[1]{\raise1pt\hbox{$\chi$}_{#1}}
\def\cu{c_1}
\def\cd{c_2}

\newcommand{\meas}{\text{meas}}

\def\dive{\text{div }}

\def\og{\leavevmode\raise.3ex\hbox{$\scriptscriptstyle\langle\!\langle$~}}
\def\fg{\leavevmode\raise.3ex\hbox{~$\!\scriptscriptstyle\,\rangle\!\rangle$}}

\definecolor{bor}{cmyk}{0.21,0.93,0.86,0.12}
\definecolor{air}{rgb}{0.178, 0.51, 0.51}
\definecolor{air2}{cmyk}{.82, 0., .67, .11}
\definecolor{range}{cmyk}{0,0.599,1,0.188}
\definecolor{range2}{rgb}{0.599,1,0.188}
\definecolor{vio}{rgb}{.4,0,.4}
\definecolor{pan}{rgb}{.17,.87,.64}
\definecolor{seagreen}{rgb}{0.18, 0.55, 0.34}

\definecolor{dblu}{rgb}{0,.12,.39}
\definecolor{amethyst}{rgb}{0.6, 0.4, 0.8}
\definecolor{blu}{cmyk}{1,.1,.1,0}
\definecolor{paleblue}{rgb}{0.69, 0.93, 0.93}
\definecolor{palerobineggblue}{rgb}{0.59, 0.87, 0.82}

\def\eqref#1{(\ref{#1})}

\def\theequation{\arabic{section}.\arabic{equation}}

%% file: LeonoriMagliocca.bbl
\begin{thebibliography}{999}

\bibitem{ADaP}B. Abdellaoui, A. Dall'Aglio, I. Peral; \href{http://www.sciencedirect.com/science/article/pii/S0022039605000616}{\it Some remarks on elliptic problems with critical growth in the gradient} J. Differential Equations, {\bf 222} (2006),  21--62.

\bibitem{ADaP2}B. Abdellaoui, A. Dall'Aglio, I. Peral; 
{\it Regularity and nonuniqueness results for parabolic problems arising in some physical models, having natural growth in the gradient}
J. Math. Pures Appl.  {\bf 90} (2008)  242--269. 


\bibitem{ABM} A.Alvino, M.F. Betta, A. Mercaldo; {\it Comparison principle for some classes of nonlinear elliptic equations}, J. Diff. Eq., {\bf 249} (2010)   3279--3290. 



\bibitem{BaM} G.Barles, F.Murat; \href{https://link.springer.com/article/10.1007/BF00375351}{\it Uniqueness and the maximum principle for quasilinear elliptic equations with quadratic growth conditions} Arch. Rational Mech. Anal., {\bf 133} (1995),    77--101. 


\bibitem{BDL}  G. Barles, F. Da Lio, {\it On the generalized Dirichlet problem for viscous Hamilton-Jacobi equations} J. Math. Pures Appl. {\bf 83} (2004),   53--75.



\bibitem{BaP} G. Barles , A. Porretta, \href{http://www.numdam.org/article/ASNSP_2006_5_5_1_107_0.pdf}{\it Uniqueness for unbounded solutions to stationary viscous
Hamilton-Jacobi equations}, Ann. Scuola Norm. Sup. di Pisa Cl. Sci., (5), {\bf 5}  (2006),   107--136.

\bibitem{BASW1}  M. Ben-Artzi, P. Souplet , F. Weissler, \href{http://www.sciencedirect.com/science/article/pii/S0021782401012430}{\it The local theory for viscous Hamilton-Jacobi equations in Lebesgue spaces}, J. Math. Pures Appli., {\bf 81} (2002),   343--378.

\bibitem{BASW2} 
 M. Ben-Artzi, P. Souplet, F. Weissler,
\href{http://www.sciencedirect.com/science/article/pii/S0764444200886085}{Sur la non-existence et la non-unicité des solutions du problème de Cauchy pour une équation parabolique semi-linéaire}, Comptes Rendus de l'Académie des Sciences-Series I-Mathematics, 329.5 (1999), 371--376.

\bibitem{BBGGPV} P. Benilan, L. Boccardo, T. Gallou\"et, R. Gariepy, M. Pierre , J.L. V\'{a}zquez, \href{http://www.numdam.org/article/ASNSP_1995_4_22_2_241_0.pdf}{\it An $L^1$ theory of existence and uniqueness of solutions of nonlinear elliptic equations}, Ann. Scuola Norm. Sup. Pisa,  {\bf 22} (1995),   241--273.






\bibitem{BDMP} M.F. Betta, R. Di Nardo, A. Mercaldo, A. Perrotta, {\it 
Gradient estimates and comparison principle for some nonlinear elliptic equations}, 
Commun. Pure Appl. Anal. {\bf 14 }(2015)  897--922. 




\bibitem{BMMP} {  F. Betta, A. Mercaldo, F. Murat, M. Porzio},  \href{https://www.cambridge.org/core/services/aop-cambridge-core/content/view/E2F62027D76760F22DF919330D2F62F0/S1292811902000519a.pdf/uniqueness_of_renormalized_solutions_to_nonlinear_elliptic_equations_with_a_lower_order_term_and_righthand_side_in_l_1.pdf}{\it Uniqueness of 
renormalized solutions to nonlinear elliptic equations with
a lower order term and right-hand side in $L^1(\Omega)$.
A tribute to
J. L. Lions},  ESAIM Control Optim. Calc. Var., {\bf 8} (2002), 239--272.

\bibitem{BMMP2}   F. Betta, A. Mercaldo, F. Murat, M. Porzio,
\href{http://www.sciencedirect.com/science/article/pii/S0362546X05004426}{\it Uniqueness results for nonlinear elliptic equations with a lower order
term}, Nonlinear Anal. {\bf 63} (2005), 153--170.


\bibitem{BlM} D. Blanchard , F. Murat, \href{https://www.cambridge.org/core/journals/proceedings-of-the-royal-society-of-edinburgh-section-a-mathematics/article/renormalised-solutions-of-nonlinear-parabolic-problems-with-l1-data-existence-and-uniqueness/9402C539D76E5F35F55ACBBBB48F1F48}{\it Renormalised solutions of nonlinear parabolic problems with $L^1$ data: existence and uniqueness}, Proceedings of the Royal Society of Edinburgh,   {\bf 127}  (1997), 1137-1152.


\bibitem{BlP} D. Blanchard , A. Porretta, \href{http://www.mat.uniroma2.it/~porretta/papers/Blanchard-Porretta-JDE.pdf}{\it Stefan problems with nonlinear diffusion and convection}, J. Diff. Eq., {\bf 210} (2005),   383--428. 


\bibitem{BDGO}  L. Boccardo, A. Dall'Aglio, T. Gallou\"et , L. Orsina, \href{http://www.sciencedirect.com/science/article/pii/S0022123696930402}{\it Nonlinear Parabolic Equations with Measure Data}, J. Func. An.,  {\bf 147} (1997),   237--258.



\bibitem{CIL} M. Crandall, H. Ishii, P.L. Lions, {\em User's guide to
viscosity solutions of second order partial differential equations}, Bulletin of the American Mathematical Society (1992), Volume 27, Number 1.




%\bibitem [Di]{Di} E. DiBenedetto, {\it  Degenerate Parabolic Equations}.


%\bibitem[Dib]{DIB}  E. DiBenedetto, \emph{$C\sp{1+\alpha }$ local regularity of weak solutions of degenerate elliptic equations}, Nonlinear Anal. {\bf 7} (1983),     827--850.

\bibitem{DNFG}R. Di Nardo, F. Feo, O. Guib\'e, 
{\it Uniqueness of renormalized solutions to nonlinear parabolic problems with lower-order terms}, 
Proc. Roy. Soc. Edinburgh Sect. A {\bf143} (2013),  1185--1208. 


\bibitem{Feo} F. Feo, \href{http://link.springer.com/article/10.1007\%2Fs11587-014-0210-z}{\it A remark on uniqueness of weak solutions for some classes of parabolic problems}, Ric. Mat., {\bf 63}  (2014),   S143--S155

%\bibitem[GT]{GT} D. Gilbarg,  N. Trudinger;  {\it Partial Differential  equations of Second Order}, 2nd ed., Springer--Verlag, Berlin/New-York (1983).

\bibitem{GMP1} N. Grenon, F. Murat , A. Porretta, {\it  Existence and a priori estimate for elliptic problems 
with subquadratic gradient dependent terms}, C. R. Acad.  Sci. Paris, Ser. I, {\bf  342} (2006), 23--28.
 
\bibitem{GMP2} N. Grenon, F. Murat , A. Porretta, \href{http://annaliscienze.sns.it/index.php?page=Article&id=300}{\it A priori estimates and existence for elliptic equations with gradient dependent terms}, Ann. Sc. Norm. Super. Pisa Cl. Sci., {\bf 13}   (2014),  137--205.


%\bibitem[LP3]{LP3}  T. Leonori, A. Porretta,   {\it Large solutions and gradient bounds for quasilinear elliptic equations},  Comm. Partial Differential Equations  {\bf 41} (2016) 952--998. 


\bibitem{LP}  T. Leonori, A. Porretta,   {\it On the comparison principle for unbounded solutions of elliptic equations with first order terms},  Journal of Mathematical Analysis and Applications, in press. 


\bibitem{LPR} T. Leonori, A. Porretta, G. Riey, {\it Comparison principles for p-Laplace equations
with lower order terms}, To appear in Ann. Mat. Pura Appl. 

% \bibitem[Lieb1]{Li} G.M Lieberman, \emph{Boundary regularity for solutions of degenerate elliptic equations}, Nonlinear Anal. {\bf 12} (1988), 1203--1219.



%\bibitem[Li2]{L1}   P.-L. Lions, {\it R\'esolution de probl\`emes elliptiques quasilin\'eaires}, Arch. Rat. Mech. An. {\bf 74} (1980), 234--254.



\bibitem{Ma} M. Magliocca, {\it Existence results for a Cauchy-Dirichlet parabolic problem with a repulsive gradient term}, Preprint. 


\bibitem{Me} A. Mercaldo, {\it
A priori estimates and comparison principle for some nonlinear elliptic equations},    In: Magnanini R., Sakaguchi S., Alvino A. (eds) Geometric Properties for Parabolic and Elliptic PDE's. Springer INdAM Series, vol 2. (2013) Springer, Milano

\bibitem{pe} F. Petitta,  {\it Renormalized solutions of nonlinear parabolic equations with general measure data}, Ann. Mat. Pura Appl., {\bf 187} (2008),  563--604.

\bibitem{PPP} F. Petitta, A. Ponce , A. Porretta, %\href{http://link.springer.com/article/10.1007/s00028-011-0115-1}
{\it Diffuse measures and nonlinear parabolic equations}, J. Evol. Eq., {\bf  11}  (2011),   861--905.


\bibitem{P1} A. Porretta, \href{http://link.springer.com/article/10.1007/BF02505907}{\it Existence results for nonlinear parabolic equations via strong convergence of truncations}, Ann. Mat. Pura e Appl.,  {\bf 177}  (1999) 143--172.

\bibitem{Po} A. Porretta, {\it  On the comparison principle for p-Laplace type operators with first order terms}, in \lq\lq On the notions of solution to nonlinear elliptic problems: results
and developments\rq\rq, 459--497, Quad. Mat. 23, Dept. Math., Seconda Univ. Napoli, Caserta (2008)

%\bibitem[Po3]{Po2} A. Porretta; {\it  The "ergodic limit'' for a viscous Hamilton-Jacobi equation with Dirichlet conditions}, Atti Accad. Naz. Lincei Cl. Sci. Fis. Mat. Natur. Rend. Lincei (9) Mat. Appl. 21 (2010), 59--78.

\bibitem{ST} G. Stampacchia, Le probl\'eme de Dirichlet pour les \'equations elliptiques du second ordre \'a coefficients discontinus, Ann. Inst. Fourier (Grenoble), {\bf 15 } (1965), 189--258. 



\end{thebibliography}
